\documentclass{amsart}

\usepackage{amsmath}
\usepackage{amssymb}
\usepackage{latexsym}
\usepackage{enumerate}
\usepackage{amscd}
\usepackage{inputenc}
\usepackage[all]{xy}
\usepackage{graphics}
\usepackage{epsfig}
\usepackage{accents}
\usepackage[active]{srcltx}
\usepackage{tikz}
\usetikzlibrary{matrix}

\usepackage{color}

\DeclareMathOperator{\HH}{H}
\newcommand{\Fu}{\operatorname{F}} 

\newcommand{\R}{\mathbb{R}}
\newcommand{\C}{\mathbb{C}}
\newcommand{\Z}{\mathbb{Z}}

\newcommand{\cH}{\mathcal{H}}

\newcommand{\fg}{\mathfrak{g}}
\newcommand{\fsl}{\mathfrak{sl}}
\newcommand{\fsu}{\mathfrak{su}}
\newcommand{\fb}{\mathfrak{b}}
\newcommand{\ft}{\mathfrak{t}}
\newcommand{\fut}{\mathfrak{ut}}

\newcommand{\fh}{\mathfrak{h}}

\newcommand{\cur}[1]{\widetilde{#1}}

\newtheorem{definition}{Definition}
\newtheorem{lemma}{Lemma} 
\newtheorem{proposition}{Proposition} 
\newtheorem{remark}{Remark}
\newtheorem{theorem}{Theorem} 
\newtheorem{corollary}{Corollary} 
 
\newtheorem{notation}{Notation}

\newcommand{\tr}{\operatorname{tr}}

\renewcommand{\hom}{\operatorname{Hom}}

\begin{document}

\title[Volumes of representations]{Volumes of $\mathrm{SL}_n(\C)$-representations  of  hyperbolic 3-manifolds}

\author{Wolfgang Pitsch}
\address{Universitat Aut\`onoma de Barcelona \\ Departament de Matem\`atiques\\
E-08193 Bellaterra, Spain \\ and BGSMATH}
\email{pitsch@mat.uab.es}
\thanks{First author was supported by Mineco grant  MTM2016-80439-P}

\author{Joan Porti}
\address{Universitat Aut\`onoma de Barcelona \\ Departament de Matem\`atiques\\
E-08193 Bellaterra, Spain\\ and BGSMATH}
\email{porti@mat.uab.cat}
\thanks{Second author was supported by Mineco grant MTM2015--66165-P }



\begin{abstract}
Let $M$ be a compact oriented three-manifold whose interior is hyperbolic  of finite volume.
We prove a variation formula for the volume on the 
variety of representations of $M$ in $\operatorname{SL}_n(\mathbb C)$.
Our proof follows the strategy of Reznikov's rigidity when $M$ is closed, in particular we use Fuks'
approach to variations by means of Lie algebra cohomology.
When $n=2$,  we get back Hodgson's
 formula for variation of  volume on the space of hyperbolic Dehn fillings.
 Our formula also yields the variation of volume on the space of
  decorated triangulations
  obtained by \cite{BFG} and \cite{DGG}.
 \end{abstract}
\maketitle

\section{Introduction}\label{sec:intro}

Let $M$ be a compact oriented three-manifold whose interior admits a complete hyperbolic metric of finite volume. 
There is a well defined notion of  
volume of a representation of its fundamental group  $\pi_1 {M}$
in $ \mathrm{SL}_n(\mathbb C) $, see Definition~\ref{def:volrepborl} for instance, and here we view the volume as a function defined on the variety of representations $\hom(\pi_1 M, \mathrm{SL}_n(\mathbb C) )$.
Bucher-Burger-Iozzi~ \cite{BBIarXiv14} have shown 
that the volume is maximal precisely
at the composition of the lifts of the holonomy with the irreducible representation
$ \mathrm{SL}_2(\mathbb C)\to \mathrm{SL}_n(\mathbb C) $.
If $M$ is furthermore closed, then this volume function is constant on connected components 
of $\hom(\pi_1 M, \mathrm{SL}_n(\mathbb C) )$
(see~\cite{MR1412681}) but  in the non-compact case the volume can vary locally. 
When $n=2$ this variety of representations (up to conjugation) contains the space of hyperbolic structures on the manifold, 
and the volume has been intensively studied in this case, starting with the seminal work of Neumann and Zagier \cite{NZ};  
in particular, a 
variation formula was obtained in Hodgson's thesis
\cite[Chapter~5]{Hodgson}, by means of Schl\"afli's variation formula for 
polyhedra in hyperbolic space. The variation of the volume was also discussed  
in \cite{BFG} when $n=3$, and in \cite{DGG} for general $n$,
through the study of decorated ideal triangulations of 
 manifolds.  

The purpose of this paper is to produce an infinitesimal formula for the  variation of the volume in 
$\hom( \pi_1 {M}, \mathrm{SL}_n(\mathbb C) )$ for arbitrary $n$ and for differentiable 
deformations of any representation, independently of the existence of 
decorated triangulations. The variety of representations has deformations that are nontrivial up to conjugation;
more precisely  the component of  $\hom( \pi_1 {M}, \mathrm{SL}_n(\mathbb C) )/ \mathrm{SL}_n(\mathbb{C})$ 
that contains the representation of maximal volume has dimension $(n-1)k$
\cite{MR2993065}, where $k$ is the number of components of $ \partial M$.
Our results are proved in  $\hom(\pi_1 M,\mathrm{SL}_n(\mathbb C))$, but they apply with no change to $\hom(\pi_1 M,\mathrm{PSL}_n(\mathbb C))$.

The boundary $\partial M$  of ${M}$   
consists of $k\geq 1$ tori, $T^2_1,\ldots, T^2_ k$.
Fix the orientation of $\partial {M}$ corresponding to the outer normal, as in 
Stokes theorem, and choose $l_i,m_i$ ordered generators of $\pi_1(T^2_i)$, so that 
if we view them as oriented curves, they generate the induced orientation. For 
instance, for the exterior of an oriented knot in $S^3$, we can take $l_1$ as a 
longitude and $m_1$
as a meridian, with $l_1$ following the orientation of the knot and $m_1$ as 
describing the positive sense of rotation. For a complex number $z \in 
\mathbb{C}$, denote by $\Re(z)$ and $\Im(z)$ its real and imaginary parts 
respectively. Assume now $\rho_t$ is a differentiable path of representations in $\hom(\pi_1 
{M}, \mathrm{SL}_n(\mathbb C) )$ parametrized by $t\in I\subset \R$.  
As a consequence of the  Lie-Kolchin theorem, there exist $1$-parameter families 
of matrices $A_i(t)\in \mathrm{SL}_n(\mathbb{C})$ and  of  upper triangular 
matrices  $a_i(t), b_i(t)\in\mathfrak{sl}_n(\mathbb{C} ) $ so that
\begin{equation}
\label{eqn:LieKolchin}
\rho_t(l_i)= A_i(t) \exp( a_i(t))A_i(t)^{-1}\qquad\textrm{ and }\qquad \rho_t(m_i)= A_i(t) \exp( b_i(t))A_i(t)^{-1}.
 \end{equation}

Our main result states:

\begin{theorem}
 \label{Theorem:main}
 Assume that $A_i(t)$, $a_i(t)$, and $b_i(t)$ as in \eqref{eqn:LieKolchin} are differentiable. Then the volume is differentiable and
 $$
\frac{d\phantom{t}}{dt}\operatorname{vol} ({M},\rho_t)= \sum_{i=1}^k \tr( \Re( b_i)\, \Im (\dot a_i) - \Re (a_i) \, \Im (\dot b_i) ).
$$
\end{theorem}

For $n=2$ this formula is precisely Hodgson's formula in the Dehn filling space, 
and for $n=3$ it is equivalent  
to the variation on the space of decorated triangulations by 
 Bergeron, Falbel and Guillloux \cite{BFG,Guilloux},
 for $n=3$, and by Dimofte, Gabela and 
Goncharov for general $n$~\cite{DGG}. See Section~\ref{subsection:n=2} below.

The hypothesis on differentiability of $A_i(t)$, $a_i(t)$, and $b_i(t)$ in Theorem~\ref{Theorem:main} is necessary, as the volume form is not differentiable on
$\hom(\pi_1 M, \mathrm{SL}_n(\mathbb C) )$, see Lemma~\ref{lem:volnondiff} below. Notice that the volume formulas of 
\cite{NZ, BFG,DGG} are defined in spaces of 
decorated triangulations, these are not open subsets of 
$\hom(\pi_1 M, \mathrm{SL}_n(\mathbb C) )/ \mathrm{SL}_n(\mathbb C) ) $ but  rather branched coverings of it.
A decoration yields a choice of Borel subgroup containing the representation of the peripheral subgroup, 
thus a differential path of decorated triangulations implies differentiability of the terms in
\eqref{eqn:LieKolchin}. In the appropiate  context, the choice
of Borel subgroups
amounts to work in the
so called augmented variety of 
representations \cite{MR3355211}.

%

Our argument is a generalization of  Reznikov's proof of the rigidity of the 
volume for closed manifolds \cite{MR1412681}. At the heart of Reznikov's 
argument is the fact that the volume of a representation $\rho$ can be seen as a 
characteristic class of the horizontal foliation on the total space of the flat 
principal bundle on ${M}$ induced by $\rho$.  This characteristic class comes from 
a cohomology class of the  Lie algebra 
$\mathfrak{g}=\mathfrak{sl}_n(\mathbb{C})$,
i.e.~it is induced by a class in $\HH^3(\mathfrak{g})$. The study of the variation 
of this characteristic class then relies on results by Fuks in \cite{MR874337}, he shows in particular
 that the variation of volume itself  can  be interpreted as a 
characteristic class of a foliation and this class stems from a cohomology 
class 
in $\HH^2(\mathfrak{g};\mathfrak{g}^\vee)$, where $\mathfrak{g}^\vee$ is the dual 
Lie algebra, viewed as a $\mathfrak{g}$-module. But since $\mathfrak{g}$ is semi-simple, this cohomology group is 
trivial, as follows by a classical result of Cartier \cite{cartier}, see 
Corollary~\ref{cor annulcohoslsldual}, hence the volume for $M$ compact is 
locally constant. We aim to follow the same outline in the non-closed case, 
which technically amounts to extend the homological tools used by Reznikov to a 
relative setting. Next we explain the plan of this work.

Firstly, in Section~\ref{sec:relcoho} we  develop the homological tools needed 
for our construction: we give a definition of cohomology groups of an object 
relative to a family of subobjects. As it is difficult to find a single place 
in the literature  where all the relative versions of the maps we need are 
explained,  we start by defining in a unified way  the relative cohomology 
constructions we will use, this is inspired by the work of Bieri-Eckmann 
\cite{MR509165} on relative cohomology of groups, but with a stronger emphasis 
on the pair object-subobject. The relative cohomology groups are devised in such 
a way  that, by definition, if $G$ is an object and $\{A\}$ is a family of 
subobjects, then the cohomology of $G$ relative to $\{A\} $ fits into a long 
exact sequence:
\[
\xymatrix{
\cdots \ar[r]& \HH^n(G;\{A\}) \ar[r] & \HH^n(G) \ar[r] &  \prod \HH^n(A) \ar[r] & \HH^{n+1}(G;\{A\})  \ar[r] & \cdots
}
\]
 We also discuss the relations between our definitions and previous existing notions of relative cohomology groups.

Secondly, in~Section~\ref{sec:relcharvar} we use the relative cohomological 
tools of the previous section to give relative versions of the constructions of 
Fuks~\cite{MR874337} on characteristic classes of foliations and variations of 
those. This gives the conceptual framework in which we can state and prove our 
formula. Up to this point we work in a general context so as to pave the way for future applications. 

In the compact case the volume of a representation $\rho: \pi_1(M) 
\rightarrow \mathrm{SL}_n(\C)$ is defined as a pull-back of a universal 
hyperbolic volume class in the continuous cohomology group 
$\HH^3_c(\mathrm{SL}_n(\C))$; since the  peripheral subgroups of a non-compact 
finite volume hyperbolic manifold are all 
abelian, the cohomology group where we want to look for a universal relative 
volume class is $\HH^3_c(\mathrm{SL}_n(\C);\{B\})$: the continuous cohomology 
groups of $\mathrm{SL}_n(\C)$ relative to the family $\{B\} $ consisting of its Borel 
subgroups. This program for constructing the volume  is carried out and explained in Section~\ref{sec:volrelatif} 
where we also  show that the definition through relative cohomology corresponds 
to the common definitions in literature, for instance the one given 
in~\cite{BBIarXiv14} via the use of the transfer map in continuous-bounded 
cohomology. 
The key point for our construction is  the crucial fact that continuous-bounded 
cohomology of an amenable group is trivial, hence   we have a canonical 
isomorphism $\HH^{3}_{cb}(\mathrm{SL}_n(\C);\{B\}) \rightarrow 
\HH^3_{cb}(\mathrm{SL}_n(\C))$ that allows to interpret the classical universal 
hyperbolic volume cohomology class as a relative cohomology class.

The study of the variation of the volume requires us then to find explicit 
cocycle representatives for the relative volume cohomology class. This is the 
object of Section~\ref{sec:varvol}. The main ingredient in this part of our 
work is the fact, underlying the van Est isomorphism connecting the continuous 
cohomology of a real connected Lie group $G$ with maximal compact subgroup $K$ 
and the  cohomology of its Lie algebra, that $\Omega^\ast_{dR}(G/K)^G$, the equivariant de Rham complex of 
the symmetric space $G/K$, computes the continuous 
cohomology of $G$; our cocycle will then appear as a bounded differential 
$3$-form on $G/K$ with a specific choice of trivialization on each Borel 
subgroup. Here we also show how to express the volume and its variation as a characteristic
class on the total space of the flat bundle induced by a representation.

Finally in Section~\ref{sec:varformula} we collect our efforts and  prove our variation 
formula and give some consequences.

{\bf Acknowledgements :} We would like to  thank Julien March\'e for helpful
 conversations on this subject. We also would like to thank the anonymous referee for suggesting many improvements to the present work.

\section{Relative cohomology}\label{sec:relcoho}
    Our approach to define relative cohomology relies on the following three crucial points:
\begin{enumerate}
 \item The existence of  functorial cochain complexes that compute the cohomology 
groups we want to relativize.
 \item The fact that given a family of objects $(A_i)_{i \in I}$ and coefficients $V_i$ and functorial 
cochain complexes $C^\ast(A_i;V_i)$, the product chain complex $\prod_{ i \in I} 
C^\ast(A_i;V_i)$ has as $n$-cohomology the product of cohomologies $\prod_{i \in 
I} \HH^n(A_i;V_i)$.
 \item The fact that the cone of a cochain map between chain complexes is  functorial in the 
homotopy  category of  complexes of $\mathbb{R}$-vector spaces. 
\end{enumerate}

\subsection{The cone construction}

Consider two cochain complexes of $\mathbb{R}$-vector spaces, i.e. differentials 
rise degree by one,  and a chain map $f: K^\ast \rightarrow L^\ast$. By 
definition $\operatorname{Cone}(f)^\ast$, the cone of $f$, is the cochain complex given by:
\[
\operatorname{Cone}(f)^n = L^{n-1}\oplus K^{n} \textrm{ and } d_{\operatorname{Cone}(f)} =  \left(  \begin{matrix} 
-d_L & f \\ 0 & d_K \end{matrix}\right).
\] 
where $d_{\operatorname{Cone}(f)}$ acts on column vectors.

One checks  that, as expected in any reasonable definition of a \emph{relative} 
cocycle, an element $\left( \begin{matrix}
 l \\ k
\end{matrix}
\right) \in L^{n-1}  \oplus K^n $ is an $n$-cocycle if and only if $k$ is a 
cocycle in $K^n$ and $d_L(l)= f^n(k)$.  For such a pair we will call $k$ the 
\emph{absolute} part and $l$ the \emph{relative} part.

This construction is functorial in the following sense.
If we have a commutative square of maps of chain complexes:
\[
\xymatrix{
K \ar[r]^f \ar[d]_r & L \ar[d]^s \\
A \ar[r]_g & B
}
\]
then we have an induced chain map: $\operatorname{Cone}(r,s)\colon \operatorname{Cone}(f)^\ast \rightarrow \operatorname{Cone}(g)^\ast$, given by 
$$\operatorname{Cone}(r,s) = \begin{pmatrix}
s & 0 \\
0 & r
\end{pmatrix}
$$
The main use of $\operatorname{Cone}(f)^\ast$ is that its homology interpolates between that of $L$ and that of $K$; indeed by construction there is a short exact sequence of complexes:
\[
 0 \rightarrow L[-1] \rightarrow \operatorname{Cone}(f) \rightarrow K \rightarrow 0,
\]
where $L[-1]$ is the shifted complex $L[-1]^n =L^{n-1}, d_{L[-1]} = -d_L$. This sequence splits in each degree and by standard techniques gives rise to a long exact sequence in cohomology.
\[
\xymatrix@R=10pt@C-6pt{
\cdots \ar[r] & \HH^{n-1}(L) \ar[r] & \HH^{n}(\operatorname{Cone}(f)) \ar[r] & \HH^{n}(K) \ar[r]^\delta & \HH^n(L) \ar[r] & \cdots
}
\]

One checks directly by unwinding the definitions that the connecting homomorphism  $\delta: \HH^{\ast}(K) \rightarrow \HH^\ast(L)$ coincides with $\HH^\ast(f)$. As expected, if we are given a morphism $(r,s)$ between maps of cochain complexes, then we will have an induced commuting ladder in cohomology:
\[
\xymatrix@R=10pt@C-6pt{
\cdots \ar[r] & \HH^{n-1}(L) \ar[r] \ar[dd]^{\HH^{n-1}(s)} & \HH^{n}(\operatorname{Cone}(f)) \ar[r] \ar[dd]^{\HH^{n}(\operatorname{Cone}(r,s))} & \HH^{n}(K) \ar[r]^\delta \ar[dd]^{\HH^{n}(r)} & \HH^n(L) \ar[dd]^{\HH^{n}(s)} \ar[r] & \cdots \\
 & & & & & \\
\cdots \ar[r] & \HH^{n-1}(B) \ar[r] & \HH^{n}(\operatorname{Cone}(g)) \ar[r] & \HH^{n}(A) \ar[r]^\delta & \HH^n(B) \ar[r] & \cdots
}
\]

\subsection{Relative cohomology}

\begin{definition}\label{def relhomology}

Let $\HH^\ast$ be our cohomology theory, possibly with coefficients (e.g. discrete group cohomology). If the cohomology theory admits
coefficients we assume that the functorial cochain complexes computing  the cohomology with coefficients are functorial in both variables. 
 
   Let $G$ be an object (Lie algebra, Lie group, manifold etc.) and $(A_i)_{i \in I}$ a family of subobjects, possibly with repetitions. If the theory admits coefficients we consider also a coefficient $V$ for the object $G$, coefficients $W_i$ for each object $A_i$ and maps between coefficients compatible with the inclusions $A_i \hookrightarrow G$, so that we have an induced map for each $i \in I$:   $C^\ast(G;V) \rightarrow C^\ast(A_i;W_i).$

  Then we define the relative cohomology of $G$ with coefficients in $V$ with 
respect to  $\{A_i\},\{W_i\}$ and we denote by  $\HH^\ast(G,\{A_i\};V,\{W_i\})$, 
the cohomology of the cone of the canonical map $C^\ast(G,V) \rightarrow  
\prod_{i\in I}C^\ast(A_i,W_i)$.
\end{definition}
As usual, if both coefficients $V$ and the $W_i$'s are the ground field $\mathbb{R}$, then we simply write 
$\HH^\ast(G,\{A_i\})$ for the relative cohomology group.
Concretely, a relative $n$-cocycle in $C^n(G,\{A_i\}; V, \{W_i\})$ is a pair 
$(c,\{a_i\}_{i \in I})$ where $c$ is an ordinary $n$-cocycle for $G$ with 
coefficients in $V$  which is a coboundary (i.e. trivial) on each subobject 
$A_i$ when the coefficients are restricted to $W_i$, together with a specific 
trivialization $a_i$ on each subobject $A_i$.

The following properties of the relative cohomology groups are immediate from 
the  functoriality of the cochain complexes $C^\ast(G;V)$ and $C^\ast(A_i;W_i)$ 
and that of the cone construction:

\begin{proposition}\label{prop:proprelcohogrp}
\begin{enumerate}
\item The relative cohomology groups $\HH^\ast(G,\{A_i\};V,\{W_i\})$  are functorial in both pairs $(G, ( A_i)_{i \in I})$  and $(V,\{W_i\})$.
\item The relative  cohomology groups fit into a long exact sequence:
\begin{multline*}
 \cdots\longrightarrow  \prod_{i \in I} \HH^{n-1}(A_i;W_i) \overset\delta\longrightarrow  \HH^{n}(G,\{A_i\};V,\{W_i\})  \longrightarrow \HH^{n}(G;V)  \\
 \longrightarrow \prod_{i \in I} \HH^{n}(A_i,W_i) \longrightarrow\cdots
\end{multline*}

\item If $J \subset I$ is a subset of the indexing family for the subobjects $A_i$, then we have an  induced natural transformation  in relative cohomology 
$$\HH^\ast(G,\{A_i\}_{i \in I};V,\{W_i\}_{i \in I}) \rightarrow \HH^\ast(G,\{A_i \}_{i \in J};V,\{W_i\}_{i \in J}).$$
\end{enumerate}
\end{proposition}


\subsection{Examples}

The different objects and cohomologies we have in mind are:
\begin{enumerate}
 
 \item Continuous or smooth cohomology of a Lie group. Here $G$ is a Lie group, for our purposes
 $\mathrm{SL}_n(\C)$, and each $A_i$ is a closed subgroup, for us a  Borel subgroup of $G$. 
 We take for $C^\ast(G;\R)$ the continuous or smooth \emph{normalized}  bar 
resolution $C^\ast_c(G;\R)$ or $C^\ast_\infty(G;\R)$ \cite[Chap. IX]{MR554917}. 
In this case, by a classical result of Hoschild-Mostow,  the canonical inclusion 
map $C^\ast_\infty(G;\R) \rightarrow C^\ast_c(G;\R)$ is a quasi-isomorphism.
 
 Another functorial way to compute the cohomology underlies  van Est theorem 
(see Section~\ref{subsec volascharac}). Given a semi-simple Lie group with 
associated symmetric space $G/K$, then the subcomplex of the de Rham complex of 
$G$-invariant differential forms computes the continuous cohomology of $G$: 
$\HH^\ast(\Omega_{dR}(G/K)^K) \simeq \HH^\ast_c(G;\mathbb{R})$. This resolution is 
functorial in the category of pairs semi-simple Lie group -- maximal compact 
subgroup.  
 \item Continuous-bounded cohomology. Since the only case we are interested in 
is  for  Lie groups with trivial coefficients we may use the cochain complex of 
continuous-bounded functions $C_{cb}^\ast(G;\R)$. 
 \item Cohomology of discrete groups. Here $G$ is a discrete group, $A_i$ a 
family of subgroups and $C^\ast(G;\R)$ stands for the usual bar resolution. This 
can of course be viewed as a particular case of continuous cohomology.
 \item De Rham cohomology of manifolds. In this case $G$ is a smooth manifold, 
$A_i$ a family of smooth submanifolds, typically the connected components of 
the boundary, and $C^\ast(G;\R)= \Omega^\ast_{dR}(G)$ is the de Rham complex of 
smooth differential forms on $M$.
 \item Lie group cohomology \cite[Chap. 7]{MR1269324}. Here $G$ and $A$ are 
respectively a real Lie algebra and a family of Lie subalgebras.  For 
$C^\ast(G;\mathbb{R})$ we use the so-called \emph{standard resolution} of 
Chevalley-Eilenberg. It is only in this case that we will need to consider non 
trivial coefficients. 
\end{enumerate}
For some of these theories one can find in the literature other relative 
cohomology theories, and the one here presented coincides with these except for one  
important case: Lie algebra cohomology. Let us review briefly this.

\medskip

\noindent{\bf Relative cohomology for discrete groups:} This has been defined by  
Bieri-Eckmann in \cite{MR509165}. Their construction defines the relative 
cohomology $\HH^\ast(G,\{A_i\})$  as the absolute  cohomology of the group $G$ 
with coefficients in a specific non-trivial $G$-module. As they explain in 
\cite[p. 282]{MR509165}, their construction is isomorphic to ours, up to a sign 
in the long exact sequence of the pair $(G;\{A_i\}_{i \in I})$. Fortunately for 
us this gives in our case a reformulation of their geometric interpretation of 
relative group cohomology  without sign problems:

\begin{theorem}\cite[Thm 1.3]{MR509165}\label{thm:comgpsdeRham}
 Let $(X,Y)$ be an Eilenberg-MacLane pair $K(G,\{A_i\};1)$. Then the relative 
cohomology sequences of $X$ modulo $Y$ and of $G$ modulo $\{A_i\}_{i \in I}$ are 
isomorphic. More precisely one has a commuting ladder with vertical 
isomorphisms:
\[
\xymatrix{
\cdots \ar[r] & \HH^n(G,\{A_i\}) \ar[r] \ar[d] & \HH_{dR}^n(G) \ar[r] \ar[d] & \Pi_i \HH^{n}(A_i) \ar[r] \ar[d]  & \HH^{n+1}(G,\{A_i\}) \ar[r] \ar[d] &  \cdots \\
\cdots \ar[r] & \HH^n(X,Y) \ar[r] & \HH^n(X) \ar[r] & \HH^{n}(Y) \ar[r] & \HH^{n+1}(X,Y) \ar[r] &  \cdots
}
\]
Where the cohomology in the bottom is the usual long exact sequence in singular cohomology.
\end{theorem}

We will be particularly interested in the case where $X =M$ is a 
manifold whose interior is  hyperbolic of finite volume, and $Y =\partial M$ is its boundary, in which case 
$Y$ is a finite disjoint union of tori, i.e. $K(\mathbb{Z}^2,1)$'s.

\medskip

\noindent{\bf De Rham cohomology:} Given a manifold $M$ and a smooth
submanifold $A$, a usual way to define relative cohomology groups 
$\HH_{dR}^\ast(M,A)$ is to consider the kernel $\Omega_{dR}^\ast(M,A)$ of the 
canonical map between de Rham complexes induced by the inclusion: 
$\Omega_{dR}^\ast(M) \rightarrow \Omega_{dR}^\ast(A)$.  This gives rise to a 
level-wise split short exact sequence of complexes:
\[
\xymatrix{
0 \ar[r] & \Omega_{dR}^\ast(M,A) \ar[r] & \Omega_{dR}^\ast(M) \ar[r] & 
\Omega_{dR}^\ast(A) \ar[r] & 0
}
\]
where the surjectivity uses the tubular neighborhood to extend any differential 
form on $A$ to a form on $M$. As these are chain complexes of 
$\mathbb{R}$-vector spaces, the usual argument based on the snake lemma gives 
rise to a long exact sequence:
\[
\xymatrix{
\cdots \ar[r] & \HH_{dR}^n(M,A) \ar[r] & \HH_{dR}^n(M) \ar[r] & \HH_{dR}^{n}(A) 
\ar[r] & \HH^{n+1}_{dR}(M,A) \ar[r] &  \cdots
}
\]

The relative de Rham cohomology can also be defined using the cone construction, cf.~\cite{BottTu}.
There is a canonical map 
\[
\Omega_{dR}^\ast(M,A) \rightarrow \operatorname{Cone}(\Omega_{dR}^\ast(M) 
\rightarrow \Omega_{dR}^\ast(A)),
\]
it maps a differential form $\omega$ of degree $n$  that is zero on $A$ to 
$(0,\omega) \in \Omega_{dR}^{n-1}(A) \oplus \Omega_{dR}^n(M)= 
\Omega_{dR}(M,\{A\})$. It is immediate to check that this  is a map of chain 
complexes, compatible with the restriction map and the connecting homomorphisms, 
and hence gives a commutative ladder;
\[
\xymatrix{
\cdots \ar[r] & \HH_{dR}^n(M,A) \ar[r] \ar[d] & \HH_{dR}^n(M) \ar[r] \ar@{=}[d] & 
\HH_{dR}^{n}(A) \ar[r] \ar@{=}[d] & \HH^{n+1}_{dR}(M,A) \ar[r] \ar[d]&  \cdots \\
\cdots \ar[r] & \HH_{dR}^n(M,\{A\}) \ar[r] & \HH_{dR}^n(M) \ar[r] & \HH_{dR}^{n}(A) 
\ar[r] & \HH^{n+1}_{dR}(M,\{A\}) \ar[r] &  \cdots
}
\]
where in the bottom we denote by $\HH^\ast_{dR}(M,\{A\})$ "our" relative 
cohomology groups. Applying the five lemma we conclude  that this canonical map 
is a quasi-isomorphism.

\medskip

\noindent{\bf Lie algebra cohomology:} This is an important case where our 
relative groups do not coincide with the usual ones. 
Given a Lie algebra $\mathfrak{g}$ and a subalgebra $\mathfrak{h}$, 
Chevalley-Eilenberg~\cite{MR1269324}  define the relative Lie algebra cohomology 
via the (now known as) relative Chevalley-Eilenberg complex:
\[
\textstyle{
\HH^\ast(\mathfrak{g},\mathfrak{h})= 
\HH^\ast(\operatorname{Hom}_{\mathfrak{h}-mod}( {\bigwedge}^\ast\mathfrak{g}
/\mathfrak{h},\mathbb{R})).}
\]

If $\mathfrak{g}$ and $\mathfrak{h}$ are Lie algebras of a Lie group $G$ and a 
closed subgroup $H$, the relationship between the cohomologies 
$\HH^\ast(\mathfrak{g})$, $\HH^\ast(\mathfrak{h})$ and 
$\HH^\ast(\mathfrak{g},\mathfrak{h})$ parallels the relationship between the 
cohomologies of the spaces in the fibration sequence:
\[
H \rightarrow G \rightarrow G/H.
\]

In particular there is a Hochshild-Serre spectral sequence relating these 
cohomologies, in contrast with the long exact sequence in our case. 

To distinguish our definition and to avoid an unnecessary clash with standard 
notations, even in case we have a family of subobjects consisting of a single 
element, we will denote our relative version as 
$\HH^\ast(\mathfrak{g},\{\mathfrak{h}\})$ and keep the usual notation 
$\HH^\ast(\mathfrak{g},\mathfrak{h})$ for the cohomology of the complex 
$\operatorname{Hom}_{\mathfrak{h}-mod}(\bigwedge^\ast\mathfrak{g}/\mathfrak{h},\mathbb{R})$.

\subsection{The case of continuous-bounded cohomology:}
Continuous-bounded cohomology produces cohomology groups that are  naturally 
Banach spaces, and this is an important feature of the theory.  As we will have 
to consider non-countable families of subgroups, there is no hope that we could 
give some metric to our relative cohomology groups $\HH^\ast_{cb}(G,\{A_i\})$ in 
any way compatible with the metrics on the absolute  cohomology groups, for the 
space $\prod_{i \in I}\HH^\ast_{cb}(A_i)$ will not usually be metrizable.    
However we are only interested in these relative cohomology groups   as tools 
interpolating between the cohomology of a group and the cohomologies of 
subgroups in a given family and  we will not enter the subtler point of the 
metric.

\begin{notation}
As a general rule we will write cohomology with coefficients separated by 
semicolons, eg.\ $\HH^3(\mathrm{SL}_n(\C);\C^n)$, unless we are dealing with 
 the ground field $\mathbb{R}$ as coefficients, in which case we will usually omit them, and write 
$\HH^3_c(\mathrm{SL}_n(\C))$ instead of $\HH^3_c(\mathrm{SL}_n(\C);\R)$. For 
cochain complexes we will however keep the reference to the coefficients in all cases.
\end{notation}

\section{Relative characteristic and variation maps}\label{sec:relcharvar}

In this section we explain how one can``relativize" Fuks construction~\cite[Chap.~3, Par.~1]{MR874337} of a characteristic class of a foliation, and more generally of a manifold with $\mathfrak{g}$-structure, and the way he handles their variation.

\subsection{Relative characteristic classes}

Given a smooth principal $G$-bundle  $E$ and a flat connection $\nabla\in \Omega^1_{dR}(E,\fg)$ on $E$ with values in a Lie algebra $\mathfrak{g}$, the absolute characteristic class map is given on the cochain level by:
\[
\begin{array}{rcl}
\operatorname{Char}_\nabla : C^\ast(\mathfrak{g};\R) & \longrightarrow & \Omega^\ast_{dR}(E) \\
\alpha & \longmapsto & (X_1,\cdots,X_n) \leadsto \alpha(\nabla X_1, \dots, \nabla X_n).
\end{array}
\]
This construction is contravariantly  functorial in both variables $\mathfrak{g}$ and $E$; flatness of $\nabla$ implies  this is in fact a chain map, i.e. it commutes with the differentials.

Fix a family of Lie subalgebras $\{\fb\}$ of $\mathfrak{g}$ and a family of 
smooth closed submanifolds $\{A\} \subset E$, for instance the family of 
connected components of the boundary of $E$. Assume that the flat connection 
$\nabla$ on $A$ restricts to a flat connection with values in $\fb^A$, an 
element in the chosen family of Lie subalgebras. Then by functoriality of the 
map $\operatorname{Char}_\nabla$ we have for each $A \subset E$  a commutative 
diagram:
\[
\xymatrix{
 C^\ast(\mathfrak{g};\R) \ar[rr]^{\operatorname{Char}_\nabla} \ar[d] & &  \Omega^\ast_{dR}(E) \ar[d] \\
  C^\ast(\mathfrak{b}_A;\R) \ar[rr]^{\operatorname{Char}_{\nabla\vert_A}} & & \Omega^\ast_{dR}(A)
}
\]
By functoriality of the cone construction we get a  \emph{relative} characteristic classe cochain map:
\[
\xymatrix{
 C^\ast(\mathfrak{g},\{\fb\}) \ar[rr]^{\operatorname{Char}_{\nabla,\{\nabla \vert_A\}}}  & &  \Omega^\ast_{dR}(E,\{A\}).
 }
\]

\subsection{Variation of characteristic classes}\label{subsec varchaclass}
Let us again briefly recall Fuks framework in the absolute case~\cite[Chap. 3, pp. 241-246]{MR874337}. We consider a $1$-parameter family of flat connections $\nabla_t$ on a manifold $E$ with values in a \emph{fixed} Lie algebra $\mathfrak{g}$. Given a Lie algebra cohomology class $[\omega] \in \HH^\ast(\mathfrak{g};\mathbb{R})$, we  want to understand the variation of the cohomology class $\operatorname{Char}_{\nabla_t}(\omega) \in \HH^\ast_{dR}(E)$ as $t$ varies.

Assume that the connection $\nabla_t$ depends differentiably on $t$, then its 
derivative in $t=0$, denoted by $\dot{\nabla}_{0}$, is again a 
connection with values in $\mathfrak{g}$. 
The characteristic class $\operatorname{Char}_{\nabla_t}(\alpha)$ depends then 
also differentiably on the parameter $t$ and, assuming $\alpha$ is of degree $n$,   
its derivative at $t=0$ is directly computed to be the de Rham cohomology class 
of the form obtained by the Leibniz derivative rule:
\[
\operatorname{Var}_{\nabla_t}(\alpha)\colon (X_1,\dots,X_n) \mapsto \sum_{i=1}^n 
\alpha(\nabla_0X_1,\dots,\nabla_0X_{i-1}, 
\dot{\nabla}_0X_i,\nabla_0X_{i+1},\dots \nabla_0X_n).
\] 
From this we get a cochain map:
\[
\begin{array}{rcl}
\operatorname{Var}_{\nabla_t} \colon C^\ast(\mathfrak{g};\R) & \longrightarrow & \Omega^\ast_{dR}(E) \\
\alpha & \longmapsto & \operatorname{Var}_{\nabla_t}(\alpha).
\end{array} 
\]

The family of connections $\nabla_t$ can also be seen as a single connection on 
$E$ but with values in the  algebra of currents 
$$\cur{\mathfrak{g}}=\mathcal{C}^{1}( 
\mathbb{R},\mathfrak{g}).
$$
The associated characteristic class map
$$\operatorname{Char}_{\nabla_t}: \HH^\ast(\cur{\mathfrak{g}})\rightarrow 
\HH^\ast_{dR}(E) $$ 
factors the variation map in a very nice way. Consider the 
following two  cochain maps:
\begin{enumerate}
\item The map $\operatorname{var}$:
\[
\begin{array}{rcl}
\operatorname{var} \colon C^n(\mathfrak{g};\R) & \longrightarrow & C^{n-1}(\mathfrak{g};\mathfrak{g}^\vee)\\
\alpha & \longmapsto & (g_1,\dots,g_{n-1}) \leadsto \big( h \mapsto \alpha(g_1,\dots,g_{n-1},h)\big)
\end{array} 
\]
where by $\mathfrak{g}^\vee$ denotes the dual vector space $\operatorname{Hom}(\mathfrak{g},\mathbb{R})$, this is 
canonically a left $\mathfrak{g}$-module  by setting $(g\phi)(h)= -\phi([g,h])$. 
\item Fuks map~\cite[Chap. 3 p. 244]{MR874337} is a cochain map, in fact a split monomorphism: 
\[
\begin{array}{rcl}
 \Fu\colon  C^{n-1}(\mathfrak{g};\mathfrak{g}^\vee) & \longrightarrow &   C^n(\widetilde{\mathfrak{g}}) 
 \end{array} 
 \]
that sends a cochain $\alpha \in C^{n-1}(\mathfrak{g};\mathfrak{g}^\vee)$ to the cochain
 \[
 (\phi_1,\dots,\phi_n) \mapsto \sum_{i=1}^n (-1)^{n-i} \Big[\alpha(\phi_1(0),\dots,\phi_{i-1}(0), \widehat{\phi_i(0)},\phi_{i+1}(0),\dots,\phi_n(0))\Big]\big(\dot{\phi_i}(0)\big)\, .
\]
\end{enumerate}
By direct computation one shows that the following diagram of cochain maps commutes:
\[
\xymatrix{
C^n(\mathfrak{g};\R) \ar[r]^{\operatorname{var}} \ar@/_1.0pc/[drr]_{\operatorname{Var}_{\nabla_t}} 
& C^{n-1}(\mathfrak{g};\mathfrak{g}^\vee) \ar[r]^{\Fu} &  C^{n}(\cur{\mathfrak{g}};\R) \ar[d]^{\operatorname{Char}_{\nabla_t}} \\
& &  \Omega_{dR}^n(E) 
}
\]

Let us now relativize the construction above. We have fixed a relative cocycle 
$(\omega,\{\beta\}) \in C^n(\fg,\{\fb\})$, a $1$-parameter family of connections 
$\nabla_t$ on a manifold $E$, and a family of closed submanifolds $\{A\}$ in $E$. 
Assume that for each value of $t$ the restriction of $\nabla_t$ to $A$ takes 
values in the Lie subalgebra $\fb^A_t \in \{\fb \}$. Then for each value of the 
parameter $t$ we have as data a relative de Rham cocycle with absolute part:
\[
(X_1,\dots X_n) \longmapsto \omega(\nabla_tX_1,\dots,\nabla_t X_n),
\]
and relative part given on each submanifold $A$ by:
\[
(Y_1,\dots,Y_{n-1}) \longmapsto \beta_t^A(\nabla_tY_1,\dots,\nabla_tY_{n-1}).
\]
The instant variation of this class is given by computing the usual limit. For the absolute part $\omega$ we get the same result as in the non-relative case:
\[
\operatorname{Var}_{\nabla_t}(\omega)
\]
For the relative part we have to compute the limit as $t\rightarrow 0$ of:
\begin{equation}
 \label{eqn:nabla}
\Delta(\beta,t)= \frac{1}{t}\big( \beta_t^A(\nabla_tY_1,\dots,\nabla_tY_{n-1}) - \beta_0^A(\nabla_0Y_1,\dots,\nabla_0Y_{n-1}) \big) \quad (\ast)
 \end{equation}
Here we are stuck as the usual tricks that lead to a Leibniz type derivation 
formula  in this case do not work: the problem lies in the fact that the class 
$\beta^A_t$ also depends on the time $t$. To overcome this difficulty we will 
impose the following coherence condition on the connection with respect to the 
family of Lie subalgebras $\fb^A_t$:

\begin{definition}\label{def coherencecond}
Assume $\mathfrak{g}$ is the Lie algebra of a connected Lie group $G$. Consider 
on a manifold $E$ with a family of submanifolds $A$ a one parameter family of 
connections $\nabla_t$. Assume that for each $A$, the restriction 
$\nabla_0\vert_A$ lies in the Lie subalgebra $\fb^A$. We say that the 
connection varies \emph{coherently} along the submanifolds $A$ with respect to 
the family $\{\fb\}$ if and only if the following holds:
\begin{enumerate}
\item[] There is a subgroup $H \subset G$ such that  for each subspace $A$ there 
exists a differentiable $1$-parameter family of elements of $H$, $h_t$ with $h_0 = \operatorname{Id}$,  
such that for each value of the parameter $t$ the connection 
$\widetilde{\nabla}^A_t = \operatorname{Ad}_{h_t}\nabla_t\vert_A$ takes values in the Lie 
subalgebra at the origin $\fb^A$.
\end{enumerate}
\end{definition}

This condition will force us to restrict our treatment of the variation of a relative characteristic class in two ways: 
\begin{enumerate}
\item Firstly we will only consider classes whose global part is an $H$-invariant cocycle, where $H$ is the group defined above.
\item  Secondly, given a  connection that varies coherently along the submanifolds $A$ with respect to the family $\{\fb\}$, we will ask for the relative part of the cocycle to satisfy the coherence condition:
\[
\forall (Y_1,\dots, Y_{n-1}) \quad  \beta_t^A(\nabla_tY_1,\dots,\nabla_tY_{n-1}) = \beta_0^A(\widetilde{\nabla}_t^AY_1,\dots,\widetilde{\nabla}_t^AY_{n-1}) .
\]
\end{enumerate}

\begin{definition}\label{def coherenceconnection}
 We say that the characteristic class \emph{varies coherently} with the connection if the previous two conditions are satisfied.
\end{definition}

Under this assumption we can pursue the computation in \eqref{eqn:nabla} above:
\begin{eqnarray*}
\lim_{t \to 0} \Delta(\beta,t) & = & \lim_{t \to 0}\frac{1}{t}\big( \beta_t^A(\nabla_tY_1,\dots,\nabla_tY_{n-1}) - \beta_0^A(\nabla_0Y_1,\dots,\nabla_0Y_{n-1}) \\
 & = &\lim_{t \to 0} \frac{1}{t} \big( \beta_0^A(\widetilde{\nabla}_t^AY_1,\dots,\widetilde{\nabla}_t^AY_{n-1})- \beta_0^A(\widetilde{\nabla}_0^AY_1,\dots,\widetilde{\nabla}_0^AY_{n-1}) \big)
 \\  
 & =& \sum_{j=1}^{n-1} \beta_0^A(\widetilde{\nabla}_0^AY_1,\dots,\widetilde{\nabla}_0^AY_{j-1},\dot{\widetilde{\nabla}^A_0}Y_j,\widetilde{\nabla}_0^AY_{j+1}, \dots\widetilde{\nabla}_t^AY_{n-1}) \\
 &=& \operatorname{Var}_{\widetilde{\nabla}_t}(\beta^A_0)
\end{eqnarray*}
Observe that, since $\omega$ is $H$-invariant, for any vector fields $(X_1,\dots,X_n)$ on $E$ we have
\[
\omega(\nabla_t X_1,\dots,\nabla_t X_n) = \omega(\widetilde{\nabla}_tX_1,\dots,\widetilde{\nabla}_tX_n), 
\]
and in particular as differential forms
\[
d\operatorname{Var}_{\widetilde{\nabla}_t}(\beta^A_t) =j^*_A (\operatorname{Var}_{\nabla_t}(\omega)) = j^*_A (\operatorname{Var}_{\widetilde{\nabla}_t}(\omega)),
\]
where $j_A: A \hookrightarrow E$ is the  inclusion. Hence the data \[(\operatorname{Var}_{\nabla_t}(\omega),\{ \operatorname{Var}_{\widetilde{\nabla}_t}(\beta^A)\}) = \operatorname{Var}_{\nabla_t,\{\widetilde{\nabla}\}}(\omega,\{\beta\})\] is indeed a relative differential form on $(E,\{A\})$.

We will now relativize the maps $\operatorname{var}$ and $F$ involved in Fuks factorization of the map $\operatorname{Var}_{\nabla_t}$.

\begin{lemma}\label{lem cochareHmodules}
Let $G$ be a connected Lie group, $\mathfrak{g}$ its Lie algebra and $H \subset 
G$ a subgroup. Then the cochain complexes $C^\ast(\fg;\R)$, 
$C^\ast(\fg;\fg^\vee)$ and $C^\ast(\widetilde{\fg};\R)$ are cochain complexes of 
$H$-modules, where the action of $H$ is induced by its adjoint action on $\fg$. 
Moreover the maps $\operatorname{var}$ and $\Fu$ are compatible with the action 
of $H$.
\end{lemma}
\begin{proof}
This is an immediate consequence of the fact that the above chain complexes are functorial in the variable $\mathfrak{g}$ and the adjoint action is by automorphisms of Lie algebras.
\end{proof}

\begin{notation}\label{not:relcochcoplx}
Denote by $C^\ast_H(\fg;\R)$,   $C^\ast_H(\fg;\fg^\vee)$ and $C^\ast_H(\widetilde{\fg})$ the subspace of fixed points under the action of $H$ of the vector spaces $C^\ast(\fg)$, $C^\ast(\fg;\fg^\vee)$ and $C^\ast(\widetilde{\fg};\mathbb{R})$ respectively.
\end{notation}

\begin{definition}\label{def relcohcoxes}
Denote by
\[
 C^\ast_H(\fg;\{\fb\}) = \operatorname{Cone}\big(C^\ast_H(\fg;\R) \rightarrow \prod C^\ast(\fb;\R)\big) 
 \]
 \[
 C^{\ast}_H(\fg,\{\fb\};\fg^\vee,\{\fb^\vee\})= \operatorname{Cone}\big(C^\ast_H(\fg;\fg^\vee) \rightarrow \prod C^\ast(\fb;\fb^\vee)\big)
 \]
 the cones taken along the maps induced by the inclusions $\fb \rightarrow \fg$. 
\end{definition}
Notice that in the above definition we do not ask a priori the chains on the  Lie algebras $\fb$  to be invariant of any sort.

\begin{proposition}\label{prop:relvarandemb}
Via the cone construction the chain maps $\operatorname{var}$ and $\Fu$  induce relative chain maps:
\[
\operatorname{var}\colon C^\ast_H(\fg,\{\fb\}) \longrightarrow C^{\ast-1}_H(\fg,\{\fb\};\fg^\vee,\{\fb^\vee\})
\]
and
\[
\Fu\colon C^{\ast-1}_H(\fg,\{\fb\};\fg^\vee,\{\fb^\vee\}) \longrightarrow C^\ast_H(\widetilde{\fg},\{\widetilde{\fb}\})
\]
\end{proposition}
\begin{proof}
This follows from the functoriality of the cone construction and the commutativity, for any Lie subalgebra $\fb \subset \fg$,  of the following two squares:
\[
\xymatrix{
C^\ast_H(\fg;\R) \ar[r] \ar[d] & C^{\ast-1}_H(\fg;\fg^\vee) \ar[d] \\
C^\ast(\fb;\R) \ar[r]  & C^{\ast-1}(\fb;\fb^\vee) 
}
\]
and
\[
\xymatrix{
C^{\ast-1}_H(\fg;\fg^\vee) \ar[r] \ar[d] & C^{\ast}_H(\widetilde{\fg};\R) \ar[d] \\
C^{\ast-1}(\fb;\fb^\vee) \ar[r]  & C^{\ast}(\widetilde{\fb};\R) 
}
\]
\end{proof}

\begin{proposition}\label{prop:relcaahrnabla}
With the notations of Definition~\ref{def coherencecond}, the map $\operatorname{Char}_{\nabla_t}: C_H^\ast(\widetilde{\fg};\R) \rightarrow \Omega_{dR}^\ast(E)$ and $\operatorname{Char}_{\widetilde{\nabla}_t^A}: C_H^\ast(\widetilde{\fb^A};\R) \rightarrow \Omega_{dR}^\ast(A)$ induce a map in relative cohomology
\[
\operatorname{Char}_{\nabla_t,\{\widetilde{\nabla}_t^A\}} : C^\ast_H(\widetilde{\fg},\{\widetilde{\fb}^A\}) \rightarrow \Omega_{dR}^\ast(E,\{A\})
\]
which is compatible with the restrictions and inflation maps, where
\[
C^\ast_H(\widetilde{\fg},\{\widetilde{\fb}^A\})  = \operatorname{Cone}\big( C^\ast_H(\widetilde{\fg};\R) \hookrightarrow C^\ast(\widetilde{\fg};\R) \stackrel{rest.}{\longrightarrow} \prod_A C^\ast(\widetilde{\fb}^A;\R) \big)
\]
\end{proposition}
\begin{proof}
By functoriality of the cone construction, it is enough to show that for each $A$ the following diagram where the vertical maps are the restriction maps commutes:
\[
\xymatrix{
 C^\ast_H(\widetilde{\fg};\R) \ar[d] \ar[rr]^-{\operatorname{Char}_{\nabla_t}}& & \Omega_{dR}^\ast(E) \ar[d] \\
 C^\ast(\widetilde{\fb}^A;\R) \ar[rr]_-{\operatorname{Char}_{\widetilde{\nabla}_t^A}} &  & \Omega_{dR}^\ast(A) 
}
\]
which is achieved by a  trivial diagram chasing.
\end{proof}

Summing up the results in this section we have shown that:

\begin{theorem}\label{them relVarmap}
Let $(\omega,\{\beta\}) \in C^\ast_H(\fg,\{\fb\})$ vary coherently along a connection $\nabla_t$ on $E$. The variation chain map $\operatorname{Var}:C^\ast(\fg;\R) \rightarrow \Omega_{dR}^\ast(E)$ induces via the cone construction a relative variation chain map
\[
\operatorname{Var}_{\nabla_t,\{\widetilde{\nabla}_t\}}: C^\ast_H(\fg,\{\fb\}) \rightarrow \Omega_{dR}^\ast(E,\{A\})
\]
whose induced map in cohomology computes the derivative at $t=0$  of the cohomology classes $\operatorname{Char}_{\nabla_t}(\omega,\{\beta\}) \in \HH^\ast_{dR}(E,\{A\})$. 
\end{theorem}

We also have a Fuks type factorization of the variation map:

\begin{theorem}\label{them factVarmap}
The relative variation map factors as:
\[
\xymatrix{
C^\ast_H(\fg,\{\fb\}) \ar[r]^-{\operatorname{var}} \ar@/_{1pc}/[drr]_{\operatorname{Var}}&  C^{\ast-1}_H(\fg,\{\fb\}; \fg^\vee,\{\fb^\vee\})\ar[r]^-{\Fu}  &  C^\ast_H(\widetilde{\fg},\{\widetilde{\fb}\})\ar[d]^{\operatorname{Char}_{\nabla_t,\{\widetilde{\nabla}_t\}}} \\
& &  \Omega^\ast_{dR}(E,\{A\})
}
\]
and this factorization is compatible with the restriction and connecting homomorphisms.
\end{theorem}

\section{The volume of a representation}\label{sec:volrelatif}

\subsection{Set-up and notations}\label{subsec:notvolrep}

Let us start with some definitions and notations involving the structure of the 
groups $\mathrm{SL}_n(\C)$. We will regard this groups as real Lie groups. Recall 
then that for each $n\geq 2$, the group $\mathrm{SU}(n) \subset 
\mathrm{SL}_n(\C)$ is a maximal compact subgroup. Let $D_n \subset 
\mathrm{SL}_n(\C)$ denote the subgroup of diagonal matrices, then $D_n \cap 
\mathrm{SU}(n) = T$ is a maximal real torus isomorphic to 
$(\mathbb{S}^1)^{n-1}$. By definition a Borel subgroup of $\mathrm{SL}_n(\C)$  
is a maximal solvable subgroup; the Borel subgroups  are also  the stabilizers 
of complete flags in $\C^n$, the Gram-Schmidt process then shows that the 
subgroup $\mathrm{SU}(n)$ acts transitively on complete flags, and hence that 
all Borel subgroups are pairwise conjugated in $\mathrm{SL}_n(\C) $ by elements in 
$\mathrm{SU}(n)$.

We fix as our model Borel subgroup $B \subset \mathrm{SL}_n(\C)$ the subgroup of 
upper-triangular matrices. 
In particular the transitive action by conjugation of $\mathrm{SU}(n)$ on the 
set of all Borel subgroups provides  each of these with a specified choice of a 
maximal compact subgroup.
Denote by $U_n \subset B$ the subgroup of unipotent matrices; this is a normal 
subgroup and gives $B$ the structure of a semi-direct product $B = U_n \rtimes 
D_n$.

Again by the Gram-Schmidt process, the inclusion $B \hookrightarrow \mathrm{SL}_n(\C)$ induces an homeomorphism of homogeneous manifolds $B/ T \simeq \mathrm{SL}_n(\C)/\mathrm{SU}(n)$. For $n=2$, this 
symmetric space is hyperbolic space.  For normalization purposes, let us recall that  
\begin{equation}
\label {eqn:normalization}
\begin{pmatrix} e^{(l+i\theta)/2} & 0 \\ 0 & e^{-(l+i\theta)/2} \end{pmatrix}= \exp \begin{pmatrix} (l+i\theta)/2 & 0 \\ 0 & {-(l+i\theta)/2} \end{pmatrix}
\end{equation}
acts on $\mathrm{SL}_2(\mathbb{C})/\mathrm{SU}(2) \simeq \cH^3$ 
as the composition of a loxodromic isometry of translation length $l$ composed with a rotation of angle $\theta$ along the same axis, cf.~\cite[\S 12.1]{MR1937957}.
%
%
%
%
%

Let $\pi$ denote the fundamental group of $M$, the compact three-manifold with nonempty boundary, whose interior
is hyperbolic of finite volume. The $k\geq 1$ boundary components  are tori, 
 $\partial {M} = T^2_1 \sqcup T^2_2 \cdots \sqcup 
T_k^2$. For each boundary component of ${M}$ fix a path from the 
basepoint  of $M$ to the boundary, this gives us a definite choice of a 
peripheral system $P_1, \cdots,P_k$ in $\pi$, where $P_i \simeq \pi_1( T^2_ i)$.


Let's now fix a representation $\rho: \pi \rightarrow \mathrm{SL}_n(\mathbb{C})$ 
for some $n \geq 2$. Since each peripheral subgroup $P_i$ is abelian and the 
Borel subgroups of $\mathrm{SL}_n(\mathbb{C})$ are maximal solvable subgroups,  
the image of the restriction of $\rho$ to $P_i$ lies in a Borel subgroup.  Fix 
for each peripheral subgroup $P_i$ such a Borel subgroup $B_i$. 

\subsection{Some known results in bounded and continuous cohomology}\label{subsec:rappelcontcoho}

The continuous cohomology of the groups $\mathrm{SL}_n(\C)$ has a rather simple structure:
\begin{proposition}\cite{MR0051508}\label{prop:cohoconSL}
 Let $n \geq 1$ be an integer, then $\HH^\ast_c(\mathrm{SL}_n(\C))$ is an exterior algebra:
 \[
  \HH^\ast_c(\mathrm{SL}_n(\C)) = \bigwedge \langle x_{n,j} \ | \ 1 \leq j \leq n \rangle 
 \]
over so-called Borel classes $x_{n,j}$ of degree $2j+1$. These classes are 
stable, if $j_n : \mathrm{SL}_n(\C) \rightarrow \mathrm{SL}_{n+1}(\C)$ denotes 
the inclusion in the upper left corner, then for $j \leq n$, 
$j_n^\ast(x_{n+1,j}) = x_{n,j}$. 
\end{proposition}

\begin{remark}
For $\mathrm{SL}_2(\mathbb{C})$ the Borel class $x_1$ is also known as the hyperbolic volume class and 
we denote it by $\operatorname{vol}_{\cH^3}$. It is completely determined by 
stability and the requirement that on $\mathrm{SL}_2(\mathbb{C})$ it is  represented by 
the cocycle 
\[
(A,B,C,D) \mapsto \int_{(A\ast,B\ast,C\ast,D\ast)}  d\operatorname{vol}_{ \cH^3}  
\]
where $(A\ast,B\ast,C\ast,D\ast)$ denotes the hyperbolic oriented tetrahedron 
with geodesic faces spanned by  the  four images of the base  point $\ast \in 
\cH^3$ by $A,B,C$ and $D$ respectively and $d\operatorname{vol}_{ \cH^3}$ is the 
hyperbolic volume form. Notice that this cocycle  is  bounded by the maximal 
volume of an ideal tetrahedron. See for instance \cite[Section 3]{MR1649192} for a thorough discussion of volumes of hyperbolic manifolds and continuous cocycles.
\end{remark}
 
Compared to the relatively simple structure of the continuous cohomology, the  
continuous-bounded cohomology of $\mathrm{SL}_n(\C)$ is considerably more complicated and largely unknown, see Monod~\cite{MR1840942}. Nevertheless, fitting well our purposes  we have the following:

\begin{proposition}\cite{MR2077038}\label{prop:compcohoconSL}
The canonical comparison  map $\HH^3_{cb}(\mathrm{SL}_n(\C)) \rightarrow \HH^3_{c}(\mathrm{SL}_n(\C))$ is surjective. 
\end{proposition}

For continuous bounded we have  also the following crucial feature, which applies in particular to the Borel and unipotent  subgroups of $\mathrm{SL}_n(\C)$:

\begin{proposition}\label{prop:contcohoamen}
 Let $G$ denote an amenable Lie group, e.g.  abelian or  solvable, then $\HH^\ast_{cb}(G) = 0$ for $\ast >0$.
\end{proposition}

We are now ready to define the the volume of our representation $\rho: \pi \rightarrow \mathrm{SL}_n(\mathbb{C})$. The long exact sequence in continuous cohomology for the pair $(\mathrm{SL}_n(\C), \{B_i\})$, where $\{B_i\}$ stands for the family of Borel subgroups we have fixed together with Proposition~\ref{prop:contcohoamen} gives immediately:

\begin{proposition}\label{prop:isorelcoho}
 For $\ast \geq 2$, the map induced by forgetting the relative part induces an  isomorphism: \[\HH^\ast_{cb}(\mathrm{SL}_n(\C), \{B_i\}) \stackrel{\sim}{\longrightarrow} \HH^\ast_{cb}(\mathrm{SL}_n(\C)).\] 
\end{proposition}

\begin{remark}\label{rem:nonisoencont}
Under the hypothesis of Proposition~\ref{prop:isorelcoho} above, and since the groups $B_i$ are the Borel subgroups of $SL_n(\C)$, by Corollary~3 in
\cite{MR2900175}  the long exact sequence  in continuous cohomology of Proposition~\ref{prop:proprelcohogrp} splits into short exact sequences:
\[
\xymatrix{
0 \ar[r] & \prod_i\HH^{\ast-1}_c(B_i) \ar[r] & \HH^\ast_{c}(\mathrm{SL}_n(\C), \{B_i\}) \ar[r] & \HH^\ast_{c}(\mathrm{SL}_n(\C)) \ar[r] & 0 \\
}
\]
Moreover, since all Borel subgroups are conjugated, all groups $\HH^\ast_c(B_i)$ are isomorphic one to each other. However since $\HH^\ast_c(B_i) \neq 0$, for instance for $\ast=1$, we do not have in general an isomorphism as for continuous \emph{bounded} cohomology.
\end{remark}


Comparing continuous cohomology  and bounded continuous cohomology for the pair $(\mathrm{SL}_n(\mathbb{C}), \{B_i\})$ gives us a commutative diagram:

\[
\xymatrix{
 \HH^3_{cb}(\mathrm{SL}_n(\C), \{B_i\}) \ar[r]^-\sim \ar[d] & \HH^3_{cb}(\mathrm{SL}_n(\C)) \ar@{>>}[d] \\
 \HH^3_c(\mathrm{SL}_n(\C),\{B_i\}) \ar[r]  & \HH^3_c(\mathrm{SL}_n(\C)) 
}
\]
This shows that the continuous-bounded cohomology class $\operatorname{vol}_{\cH}$ has a canonical representative as a continuous  bounded  relative class $\operatorname{vol}_{\cH,\partial} \in \HH^3_c(\mathrm{SL}_n(\C),\{B_i\})$.

By construction the representation $\rho$ induces  a map of pairs $\rho\colon (\pi, \{ P_i \}) \rightarrow (\mathrm{SL}_n(\C), \{B_i\})$, hence by functoriality we have an induced map in continuous cohomology
\[
 \HH^3_c(\mathrm{SL}_n(\C), \{B_i\}) \stackrel{\rho^\ast}{\longrightarrow} \HH^3_c(\pi, \{ P_i \}).
\]
But for discrete groups continuous cohomology and ordinary group cohomology coincide, so we have a well-defined class, up to a possible ambiguity  given by the choice of the Borel subgroups $B_i$,
\[
 \rho^\ast(\operatorname{vol}_{\cH,\partial}) \in \HH^3(\pi, \{ P_i \}).
\]

\begin{proposition}\label{prop:volindepBorelchoice}
The class $\rho^\ast(\operatorname{vol}_{\cH,\partial}) \in \HH^3_c(\pi, \{ P_i \})$ is independent of the possible choice of a different family of Borel subgroups $\{B_i\}$.
\end{proposition}
\begin{proof}
Let us assume for clarity that we have two possible choices $B_j$ and $B_j'$ for the Borel subgroup that contains $\rho(P_j)$, and that we make a unique choice for the rest of the peripheral subgroups. We denote the two families of subgroups by $\{B_{i \neq j}, B_j\}$ and $\{B_{i\neq j}, B'_j\}$. Because Borel subgroups are closed in $\mathrm{SL}_n(\C)$, their intersection, as $B_i \cap B_i'$, is also amenable. The restriction of $\rho$ to the peripheral subgroup $P_j$ factors in both cases through this intersection, so we have a commutative diagram of group homomorphisms:
\[
\xymatrix{
& & (\mathrm{SL}_n(\C),\{B_{i \neq j}, B_j\})  \\
(\pi,\{P_{i\neq j},P_j\}) \ar@/^1pc/[urr] \ar@/_1pc/[drr] \ar[r] & (\mathrm{SL}_n(\mathbb{C}),\{B_{i \neq j}, B_j \cap B'_j\})  \ar[ur] \ar[dr]& \\  
& & (\mathrm{SL}_n(\C),\{B_{i \neq j}, B'_j\})  
}
\]
Together with the forgetful isomorphisms  to the absolute cohomology of $\mathrm{SL}_n(\C)$, and given  the fact that the subgroups involved are all amenable, we have a commutative diagram:
\[
\xymatrix{
	& & \HH^3_{cb}(\mathrm{SL}_n(\C),\{B_{i \neq j}, B_j\}) \ar@/_1pc/[dll] \ar[dl] \ar[d]^{\wr} \\
	\HH^3_{cb}(\pi,\{P_{i \neq j}, P_i\} )  & \HH^3_{cb}(\mathrm{SL}_n(\mathbb{C}),\{B_{i \neq j}, B_j \cap B'_j\} )  \ar[l]  \ar[r]^-\sim & \HH^3_{cb}(\mathrm{SL}_n(\C)) \\  
	& & \HH^3_{cb}(\mathrm{SL}_n(\C),\{B_{i \neq j}, B'_j\}) \ar[ul]   \ar@/^1pc/[ull] \ar[u]_{\wr}
}
\]

and this finishes the proof.
\end{proof}
Now $M$ is a $K(\pi,1)$ and each boundary component is a $K(P_i,1)$ for the 
corresponding peripheral subgroup, in particular $\HH^3(\pi, \{ P_i \}) \simeq 
\HH^3({M};\partial {M}) \simeq \R$ by 
Theorem~\ref{thm:comgpsdeRham}, due to Bieri and Eckmann,  and this leads to our compact definition of the 
volume of a representation (for a more precise statement see 
Definition~\ref{def:defvolplusprecis}):

\begin{definition}\label{def:volrepborl}
 Let $\rho\colon \pi \rightarrow \mathrm{SL}_n(\C)$ be a representation of the fundamental group of a finite volume hyperbolic $3$-manifold. Then, evaluating on our fixed fundamental class $[M,\partial M] \in \HH_3(M,\partial M)$ we set:
\[
 \operatorname{Vol}(\rho) = \langle \rho^\ast(\operatorname{vol}_{\cH,\partial}),[{M},\partial {M}]\rangle .
\]
 \end{definition}

In~\cite{MR3026348} Bucher, Burger and Iozzi prove that the volume of a 
representation  $\pi \rightarrow \mathrm{SL}_n(\C)$ is maximal 
at the composition of the irreducible representation
$\mathrm{SL}_2(\C)\to\mathrm{SL}_n(\C)$
with a lift of the holonomy.
Their definition, as 
ours, relies on continuous bounded cohomology and are clearly equivalent: their 
transfer argument is replaced here by an isomorphism through a relative 
cohomology group. The passage through continuous cohomology seems for the moment 
rather useless, it will however be crucial in our next set: the study of the 
variation of the volume.

\section{Variation of the volume class}\label{sec:varvol}

We  follow Reznikov's idea~\cite{MR1412681} to prove rigidity of the volume 
in the compact case. We will first  show that the volume class can be viewed as 
a characteristic class on the total space of the flat bundle defined by the 
representation, then find explicit relative cocycles representing 
$\operatorname{vol}_{\cH,\partial}$ and finally apply the machinery of 
Section~\ref{sec:relcharvar}.

Let us start with some more notations. In the previous section we defined a series 
of Lie subgroups of $\mathrm{SL}_n(\C)$, we now  pass to their Lie algebras, 
all viewed as real Lie algebras.

\begin{center}
\begin{tabular}{c|c| l}
Lie group & Lie algebra & Description as subgroup \\  \hline 
$\mathrm{SL}_n(\C)$ & $\fsl_n$ & \\
$\mathrm{SU}(n)$ & $\fsu_n$ & Fixed maximal compact subgroup\\
$B$ & $\fb_n$ & Fixed Borel subgroup of upper triangular matrices \\
$D_n$ & $\fh_n+i\fh_n$ & Subgroup of diagonal matrices in $B$\\
$T= D_n\cap \mathrm{SU}(n)$ & $\fh_n$ & Maximal torus in $\mathrm{SU}(n)$ (and in $B$ and $\mathrm{SL}_n(\C)$)\\
$U_n$ & $\fut_n$ & Subgroup of unipotent elements in $B$. \\
& & 
\end{tabular}
\end{center}

For explicit formulas, we will need a concrete basis for the real Lie algebra 
$\fsu_n$. Recall that $\fsu_n = \{ X \in M_n(\mathbb{C}) \ | \ X+ 
{}^t\overline{X} =0 \textrm{ and } \tr(X)=0 \}$.

There is a standard $\mathbb{R}$-basis of $\fsu_2$, orthogonal  with respect to the Killing form:
\[
h = \left(
\begin{matrix}
i/2& 0 \\
0 & -i/2
\end{matrix}
\right)
, \quad 
e = \left(
\begin{matrix}
0 & 1/2 \\
 -1/2 & 0
\end{matrix}
\right)
, \quad
f = \left(
\begin{matrix}
0 &i/2 \\
 i/2 & 0
\end{matrix}
\right).
\]

From this we can construct an analogous basis for $\fsu_n$; we only give here 
the non-zero entries of the matrices.
\begin{enumerate}
\item For an integer $1 \leq s \leq n-1$ let $h_s$ denote the matrix with a 
coefficient $i/2$ in diagonal position $s$ and a coefficient $-i/2$ in diagonal 
position $n$. It will be convenient to denote $h_{st}=h_s-h_t$.
\item For any pair of integers $1 \leq s < t \leq n$ let $e_{st}$  have 
coefficient row $s$ and column $t$ equal to $1/2$ and coefficient row $t$ and 
column $s$ equal to $-1/2$.
\item For any pair of integers $1 \leq s < t \leq n$ let  $f_{st}$ denote the 
matrix which  has coefficient row $s$ and column $t$ equal to $i/2$ and 
coefficient row $t$ and column $s$ equal to $i/2$.
\end{enumerate}

Notice that the matrices $h_s$ generate the Lie subalgebra $\fh$, the Lie 
algebra of the real torus $T$. The dual basis will be denoted by 
$h^\vee_s,e^\vee_{st},h^\vee_{st}$.
With these conventions, for $n=2$,  $h=h_1, e=e_{12}$ and $f=f_{12}$.

Analogously, for $\fb$, the Lie algebra of upper triangular matrices with zero 
trace, we have a basis made of the matrices $h_s$, $ih_s$, $1 \leq s \leq n-1$, 
and for $1 \leq k < l \leq n$, the matrices $ur_{kl}$ (upper real) which are 
equal to $1$ in row $k$ and column $l$ and $ui_{kl} = i ur_{kl}$ (upper 
imaginary matrices). We have $ur_{kl}= e_{kl}-i\, f_{kl}$ and 
$ui_{kl}=i\, e_{kl}+f_{kl}$.

The following  result provides us with the right cochain complex in which to 
find our cocycle representatives; beware that the relative cohomology of Lie 
algebras in the statement is not the one we defined in 
Section~\ref{sec:relcoho}, but the classical one as defined for instance in 
Weibel~\cite[Chap. 7]{MR1269324}.
 
\begin{proposition}[\cite{HochMostow} van Est isomorphism for trivial coefficients]
Let $G$ be a connected real Lie group. Denote by $\mathfrak{g}$ its Lie algebra and $\mathfrak{k}$ the Lie algebra of a maximal compact subgroup $K \subset G$. Then for all $m$ there is a canonical isomorphism $\HH^m_c(G;\R) \simeq \HH^m(\mathfrak{g},\mathfrak{k};\R)$.  More precisely the de Rham cochain complex of left-invariant differential forms
\[
 \xymatrix{
 0 \ar[r] & \Omega^0_{dR}(G/K)^G \ar[r] & \cdots  \ar[r] & \Omega^n_{dR}(G/K)^G \ar[r] & \cdots
 }
\]
computes both cohomologies.
\end{proposition}

Functoriality of the cone construction allows to extend van Est isomorphism to relative cohomology as follows. Fix a connected Lie group $G$ and a family of connected closed subgroups $\{B_i\}$. Pick for each index $i$ a maximal compact subgroup $K_i \subset B_i$ and fix a maximal compact subgroup $K \subset G$. Then, by maximality, for each index $i$ there is an element $g_i \in G$ such that $K_i \subset g_i Kg_i^{-1}$. Then the composite
\[
j_{g_i} : B_i/K_i \longrightarrow G/g_iKg_{i}^{-1} \stackrel{c_{g_i}}{\longrightarrow} G/K,
\]
where the second map is induced by conjugation by $g_i$, induces a cochain map:
\[
\Omega_{dR}^\ast(G/K)^G \rightarrow \Omega_{dR}^\ast(B_i/K_i)^{B_i}
\]
which in, lets say, continuous cohomology is the map induced by the inclusion $B_i \rightarrow G$. Indeed it is clear for the first map using van Est isomorphism with the maximal subgroups $g_iKg_i^{-1}$ in $G$ and $K_i$ in $B_i$, and as for the second map, by construction it induces in cohomology the map that is induced by conjugation by $g_i$ and this is well-known to be the identity.
Let us denote the first composite by $j_{g_i}: B_i/K_i \longrightarrow G/K$.
Denote respectively by $\mathfrak{g}$, $\mathfrak{k}$, $\mathfrak{b}_i$, $\mathfrak{k}_i$ the Lie algebras of $G$, $K$, $B_i$, $K_i$. The an immediate application of the five lemma and van Est isomorphism gives us:

\begin{corollary}[relative van Est isomorphism]\label{cor:relvanEst}
With the above 	notations and conventions the cone on the map
\[
\xymatrix{
\Omega_{dR}^G(G/K) \ar[rr]^-{\Pi j_{g_i}^\ast}& &  \Pi \Omega(B_i/K_i)^{B_i}
}
\]
computes both the relative continuous cohomology groups $\HH^\ast_c(G,\{B_i\};\mathbb{R})$ and the unaesthetic relative Lie cohomology  groups $\HH^\ast(\mathfrak{g},\mathfrak{k}, \{\mathfrak{b}_i,\mathfrak{k}_i\};\mathbb{R})$. In particular both these relative cohomology groups are canonically isomorphic.
\end{corollary}

Recall that the volume class comes from a bounded cohomology class, so its de Rham representative will be rather special and can be explicitly detected  thanks to the following result of Burger and Iozzi~(Prop. 3.1 in \cite{MR2322535})

\begin{proposition}\cite{MR2322535}\label{prop:boundeddeRham}
Let $G$ be a connected semi-simple Lie group with finite center, let $K$ be a maximal compact subgroup, let  $G/K$ the associated symmetric space and let $L \subset G$ be any closed subgroup. Then there exists a map
\[
\delta_{\infty,L}^\ast: \HH^\ast_{cb}(L;\mathbb{R}) \rightarrow \HH^\ast(\Omega_{dR, \infty}(G/K)^L)
\]  
such that the diagram:
\[
\xymatrix{
 \HH^\ast_{cb}(L;\mathbb{R}) \ar[r]^{c^\ast_L} \ar[dr]_{\delta_{\infty,L}} & \HH^\ast_c(L;\mathbb{R}) & \HH^\ast(\Omega_{dR}(G/K)^L) \ar_{\sim}[l] \\
 & \HH^\ast(\Omega_{dR,\infty}(G/K)^L) \ar[ur]_{i_{\infty,L}} & 
}
\]
commutes, where $\Omega_{dR,\infty}(G/K)$ is the de Rham complex of bounded differential forms with bounded differential and  $i_{\infty,L} $ is the map induced in cohomology by the inclusion of complexes $\Omega_{dR,\infty}(G/K) \hookrightarrow \Omega_{dR}(G/K)$.
\end{proposition}

\subsection{A relative cocycle representing $\operatorname{vol}_{\cH,\partial }$}\label{subsec:relcoc}
We will apply the relative van Est isomorphism in the particular case where $G= SL_n(\C)$, $K= SU(n)$ and $\{B\}$ is the family of all Borel subgroups and in cohomological degree $3$. Here the situation is simpler, as for any Borel subgroup $B \cap SU(n)$ is a maximal torus and in our case this is also a maximal compact subgroup of $B$ so the "conjugation" part of the statement can be avoided and simply use as maximal compact subgroup of $B$ the intersection $B \cap SU(n)$.

In particular, to represent the class $\operatorname{vol}_{\cH,\partial }$,   we look for a  relative cocycle whose  absolute part lies in $\Omega_{dR}^3(\mathrm{SL}_n(\mathbb{C})/\mathrm{SU}(n))^{SL_{n}(\C)}$ and whose  relative part lies in the groups $\Omega^2_{dR}(B/(B \cap SU(n)))^B$.

We take now advantage of the fact that  all pairs $(B,T)$ where $B$ is a Borel subgroup and $T$ a maximal torus in $B$  are conjugated in $SL_n(\C)$, so in fact we only need to determine the relative part  for  our standard Borel $B$ of upper triangular matrices; if $\beta$ is a relative part for this particular subgroup and $B'$ is another Borel, there exists an element $g \in SL_n(\mathbb{C})$ that conjugates $(B,T) $ and $(B',SU(n)\cap B')$, then conjugation by $g$ induces a homeomorphism $c_g: B/T \rightarrow B'/(B' \cap SU(n))$, hence the relative part for $B'$ is given by $c_{g^{-1}}^\ast(\beta)$.

\subsubsection{The absolute part}\label{subsub:absolpart}

Denote by $K^\mathbb{R}_{\fsl_n}$ the real Killing form of the real Lie algebra $\fsl_n$. With respect to this form we have an orthogonal decomposition $\fsl_n = \fsu_n \oplus i\fsu_n$. We denote by:
\[
\begin{array}{rcl}
pr_{\fsu_n}  : \fsl_n & \longrightarrow & \fsu_n \\
A & \longmapsto & \frac{1}{2}(A - {}^t\overline{A})
\end{array}
\quad \textrm{ and } \quad 
\begin{array}{rcl}
pr_{i\fsu_n}  : \fsl_n & \longrightarrow & i\fsu_n \\
A & \longmapsto & \frac{1}{2}(A + {}^t\overline{A})
\end{array}
\]
the canonical projections.

The behavior of these projections with respect to the Lie bracket is given by:
\begin{equation}\label{eq:brackreel}
pr_{\fsu}([a,b]) =  [pr_{\fsu}a,pr_{\fsu}b] + [pr_{i\fsu}a,pr_{i\fsu}b],
\end{equation}
\begin{equation}\label{eq:brackim}
pr_{i\fsu}([a,b])  =  [pr_{\fsu}a,pr_{i\fsu}b] + [pr_{i\fsu}a,pr_{\fsu}b]. 
\end{equation}
The tangent space at the class of $\operatorname{Id}$ of the symmetric space 
$\mathrm{SL}_n(\mathbb{C})/\mathrm{SU}(n)$ is canonically identified with 
$i\fsu_n$, and the induced action of $\mathrm{SU}(n)$ on this tangent space is 
easily checked to be the adjoint action. Let us now consider the following 
rescaling of the  \emph{complex} Killing form on $\fsl_n$, $A,B \leadsto 
\tr(AB)$. This gives rise to a complex valued alternating $3$ form, sometimes 
known as the (here rescaled) Cartan-Killing form: 
$CK^\mathbb{C}_{\fsl_n} \colon (A,B,C) \mapsto \tr(A[B,C])$. It is folklore knowledge 
that ``the hyperbolic volume  is the imaginary part of this Cartan-Killing form" 
(see Yoshida~\cite{Yoshida} for a precise statement when  $n=2$ or 
Reznikov~\cite{MR1412681}); let us turn this into a precise statement.
We fix our attention in the following part of the  de Rham complex:

\[
 \Omega_{dR}^2(\mathrm{SL}_n(\mathbb{C})/\mathrm{SU}(n))^{ \mathrm{SL}_n(\mathbb{C})} \rightarrow \Omega_{dR}^3(\mathrm{SL}_n(\mathbb{C})/\mathrm{SU}(n))^{ \mathrm{SL}_n(\mathbb{C})} \rightarrow \Omega_{dR}^4(\mathrm{SL}_n(\mathbb{C})/\mathrm{SU}(n))^{ \mathrm{SL}_n(\mathbb{C})}.
\]
\begin{lemma}\label{lem:2formesslninvtriv}
The vector space $\Omega^2_{dR}(\mathrm{SL}_n(\mathbb{C})/\mathrm{SU}(n))^{ \mathrm{SL}_n(\mathbb{C})}$ is trivial.
\end{lemma}
\begin{proof}
By transitivity of the action,  an alternating $2$-form on the homogeneous space $\mathrm{SL}_n(\mathbb{C})/\mathrm{SU}(n)$ is completely determined by what happens at the class of the identity, i.e. by a unique element in $ (\bigwedge^2(i\fsu_n)^\vee)^{\mathrm{SU}(n)}$. As $\mathrm{SU}(n)$-modules $i\fsu^\vee$ and $\fsu^\vee$ are isomorphic, and via the real Killing form on $\mathrm{SU}(n)$, a symmetric non-degenerate form, the Lie algebra $\fsu$ and its dual are also isomorphic $\mathrm{SU}(n)$-modules. So to prove the statement it is enough to show that $(\bigwedge^2 \fsu_n)^{\mathrm{SU}(n)} = 0$. Let $\phi\colon \mathfrak{su}(n)\wedge \mathfrak{su}(n)\to \mathbb R$ be a skew-symmetric invariant form. Invariance by the adjoint action 
of $\mathrm{SU}(n)$ is equivalent to:
\[
\phi([X,Y],Z)+ \phi(Y, [X,Z])=0 \qquad \forall X,Y,Z\in  \mathfrak{su}(n)\, .
\]
Combined with skew-symmetry of  both $\phi$ and  the Lie bracket, this equality yields
\[
\phi([X,Y],Z)= \phi( [X,Z],Y)= - \phi( [Z,X],Y)  \qquad \forall X,Y,Z\in  \mathfrak{su}(n)\, .
\]
Namely,  $\phi([X,Y],Z) $ changes the sign when the entries $X,Y,Z\in  \mathfrak{su}(n)$ are cyclically permuted, therefore it vanishes.
Then $\phi=0$ because $\mathfrak{su}(n)$ is simple.
\end{proof}

For a manifold $X$, denote by $Z_{dR}^n(X) \subset \Omega_{dR}^n(X)$ the subspace of closed forms.
\begin{corollary}\label{cor:3exactarecohomology}
The canonical quotient map $Z^3_{dR}(\mathrm{SL}_n(\C)/\mathrm{SU}(n))^{ \mathrm{SL}_n(\C)} \rightarrow \HH^3_c(\mathrm{SL}_n(\C)) \simeq \R$ is an isomorphism.
\end{corollary}

Since $\HH^3_c(\mathrm{SL}_n(\C)) \simeq \R$ by Borel's computations, there is a 
unique closed form on $\mathrm{SL}_n(\C)/\mathrm{SU}(n)$ that represents the 
class $\operatorname{vol}_{\cH}$.   There is an obvious candidate for such a form, it is given 
on the tangent space at $\operatorname{Id}$ by:
\begin{equation*}
\begin{array}{rcl}
 \bigwedge^3 i\fsu_n  & \longrightarrow  & \mathbb{R} \\
(A,B,C) & \longmapsto & 2i\operatorname{tr}(A[B,C])= -2\Im \operatorname{tr}(A[B,C]).
\end{array} 
\end{equation*}
Then $\varpi_n\colon \fsl_n\to\mathbb{R}$ is the composition of the projection 
$pr_{i\fsu}\colon \fsl_n\to i\fsu_n$ with this form:
\begin{equation*}
\begin{array}{rcl}
\varpi\colon \bigwedge^3 \fsl_n \quad & \longrightarrow  & \mathbb{R} \\
(A,B,C) & \longmapsto & 2i\operatorname{tr}(  pr_{i\fsu}(A)[pr_{i\fsu}(B),pr_{i\fsu}(C)]).
\end{array} 
\end{equation*}
That this form is alternating and invariant under the adjoint action of 
$\mathrm{SU}(n)$ is an immediate consequence of the fact that 
the Cartan-Killing form $
(A,B,C)\mapsto \tr(A[B,C])=\tr ( ABC-ACB)
$ is alternating and $\mathrm{SU}(n)$-invariant, 
and that the adjoint action of $\mathrm{SU}(n)$ respects the 
decomposition of $\fsl_n = \fsu_n \oplus i\fsu_n$. Observe that by construction 
this form is compatible with the inclusions $\fsl_n \rightarrow \fsl_{n+1}$: if 
we denote the form defined by $\fsl_n$ by $\varpi_n$ then 
$\varpi_{n+1}|_{\fsl_n}=\varpi_n$, in line of the stability result of Borel in 
degree $3$. We only have to check that this is a  cocycle when viewed as a  
classical relative cocycle in Lie algebra cohomology of $\fsl_n/\fsu_n =i\fsu_n$ 
(i.e. gives rise to a closed form), that it is not trivial and fix the 
normalization constant; this will done by comparing it with the hyperbolic 
volume form for $n=2$.

\begin{lemma}\label{lem:omegaisclosed}
The alternating $3$-form $\varpi \in \operatorname{Hom}(\bigwedge^3 \fsl_n,\R)$ is a cocycle.
\end{lemma}
\begin{proof}
By definition of the differential in the Cartan-Chevalley complex see 
Weibel~\cite[Chap. 7]{MR1269324}, and since $[i\fsu_n,i\fsu_n] \subset \fsu_n$, 
the differential  in this cochain complex is  in fact trivial, so any element in 
$ \operatorname{Hom}(\bigwedge^3 i\fsu_n;\R)$ is a cocycle. 
\end{proof}
\begin{lemma}\label{lem:varpiishypform}
Via the canonical isomorphism $\mathrm{SL}_2(\C)/\mathrm{SU}(2) \simeq \mathcal{H}^3$, the form $\varpi_2$ is mapped to the hyperbolic volume form $d\operatorname{vol}_{\mathcal{H}^3}$.
\end{lemma}
\begin{proof}
We use the half-space model $\cH^3=\{ z+t j\mid z\in\C, t\in\R, t>0\}$, 
so that the action of $\mathrm{SL}_2(\C)$ on $\mathbb{P}^1(\C)\cong\C\cup\{\infty\}$ extends 
conformally by isometries. In particular $\mathrm{SU}(2)$ is the stabilizer of the point $j$, and we use the natural map 
from $\fsl_2$ to the tangent space $T_j\cH^3$ that maps $a\in \fsl_2$ to the  vector 
$\frac{d\phantom{t}}{dt}\exp(t a) j\vert_{t=0} $. From this construction,  $\fsu_2$ is mapped to zero and 
$i\mathfrak{su}_2$ is naturally identified to
tangent space to $\cH^3$ at $j$.
Thus the form induced by the volume form is the result of composing
a form on $i\mathfrak{su}_2$ with 
the projection $\mathfrak{sl}_2(\mathbb{C})\to i\mathfrak{su}_2$. By  $\mathrm{SU}(2)$-invariance, it suffices to check that its
evaluation at an orthonormal basis is 1. The ordered basis
 \begin{equation}
  \label{eqn:basis}
  \left\{
    \begin{pmatrix} 0 & 1/2 \\
     1/2 & 0
    \end{pmatrix} , 
   \begin{pmatrix} 0 & i/2 \\
     -i/2 & 0
    \end{pmatrix} , 
      \begin{pmatrix} 1/2 & 0 \\
     0 & 1/2
    \end{pmatrix}
    \right\}
 \end{equation}
%
of $i\mathfrak{su}_2$
is mapped to $\{1, i, j\}$ via the isomorphism $i\mathfrak{su}_2\cong T_j\cH^3 $, which is a positively oriented orthonormal basis, and  
$\varpi$ evaluated at the basis \eqref{eqn:basis} is $1$.
%
%
%
%
%
\end{proof}

\begin{remark}\label{rem FormeCocylce}
The cocycle has the following precise form:
\[
 \varpi = - {\sum_{j<k}} (ih_{jk})^\vee\wedge (ie_{jk})^\vee \wedge (if_{jk})^\vee.
\]
Fixing a pair of indices $1 \leq j<k \leq n$ fixes a Lie subalgebra  in 
$\fsu_n$ isomorphic to $\fsu_2$. The restriction of $\varpi$ to each of these 
$\frac{n(n-1)}{2}$ copies of $\fsu_2$ is exactly the corresponding hyperbolic 
volume form.
\end{remark}

\begin{remark}\label{rem:CartanKilling}
The imaginary part of the Cartan-Killing form, 
 $(x,y,z)\mapsto \Im\tr([x,y]z)$ $\forall x,y,z\in\fsl_n$,  is cohomologous to $-2\varpi_n$, but it does not come from a bounded
 cocycle in $\mathrm{SL}_n\C$ (cf.~\cite[Lemma 3.1]{Yoshida} for $n=2$). 
\end{remark}

\subsubsection{The relative part}\label{subsubsec:relpart}

We now turn to the relative part of our cocycle.  For this we have 
to understand the restriction of the form $\varpi \in \Omega^3_{dR}(\mathrm{SL}_n(\mathbb{C})/\mathrm{SU}(n))$ along the canonical map $B/T_n \rightarrow \mathrm{SL}_n(\mathbb{C})/\mathrm{SU}(n)$ induced by the inclusion of an arbitrary Borel subgroup $B$. As all Borel subgroups are conjugated in $\mathrm{SL}_n(\mathbb{C})$ by an element of $\mathrm{SU}(n)$, provided by the Gram-Schmidt process, and the form $\varpi$ is $\mathrm{SU}(n)$-invariant, it is enough to treat the case of our fixed Borel $B$ of upper-triangular matrices. As we will see, because we require our trivializations to come from a bounded class there will be only one choice, and this uniqueness will then provide the coherence condition we need for computing the variation.

\begin{lemma} \label{lem:firstcohoinvB/T}
The vector space  $\Omega^1_{dR}(B/T)^B$ is generated by the closed $1$-forms $ih_s^{\vee}$. In particular, the differential $\Omega^1_{dR}(B/T)^B \rightarrow \Omega^2_{dR}(B/T)^B$ is trivial and  $\HH^1_c(B;\mathbb{R}) = \mathbb{R}^{n-1}$.
\end{lemma}
\begin{proof}
As before, by transitivity an element in $\Omega^1_{dR}(B/T)^B$ is determined by its restriction to the tangent space to the identity, $\fb_n/\fh_n$; i.e by  a form on this tangent space invariant under the induced action by the torus $T$. The Borel Lie algebra $\fb_n$, the Lie algebra of the torus $\mathfrak{h}_n$, and the Lie algebra of strictly upper triangular matrices $\fut_n$ fit into a commutative diagram with exact row  of $T$-modules:
\[
\xymatrix{
& & \fut_n \ar@{_{(}->}[d] \ar@{_{(}->}[dr] & \\
0 \ar[r] &  \mathfrak{h}_n \ar[r] & \mathfrak{b}_n \ar[r] &  \mathfrak{b}_n  /\mathfrak{h}_n \ar[r] & 0.
}
\]
We view a $T$-invariant form on $\mathfrak{b}_n  /\mathfrak{h}_n $ as a $T$-invariant form  $\psi\colon \mathfrak{b}_n \rightarrow \mathbb{R}$ which is trivial on $\mathfrak{h}_n$.
The action of $T$ is readily checked to be  induced by the conjugation action of $T$ on $B$, hence invariance is equivalent to:
\[
\forall t \in \mathfrak{h}_n, \ \forall b \in \mathfrak{b}_n, \quad \psi([t,b])= 0  .
\]
But $[\mathfrak{h}_n,\mathfrak{b}_n] = \fut_n$, hence $\psi$ is in fact a form on $\fb_n/\fut_n$.
It is finally straightforward to check that indeed the $n-1$ forms $h_s^\vee$ are both closed and linearly independent.
\end{proof}
\begin{lemma}\label{lem:trivsurborel}
The space $\Omega^2_{dR}(B/T)^B$ has a basis given by
\begin{enumerate}
\item the $\frac{n(n-1)}{2}$ forms $ur_{kl}^\vee \wedge ui_{kl}^\vee$ for all $1 \leq k < l \leq n$;
\item the $\frac{(n-1)(n-2)}{2}$ closed  forms $ih_s^\vee \wedge ih_r^\vee$ for all $1\leq s < r \leq n-1$.
\end{enumerate}
\end{lemma}
\begin{proof}
Such a form, say $\phi$, is exactly a $T$-invariant and alternating $2$-form on $\fb_n/\fh_n$. As a $T$-module, $\fb_n/\fh_n = i\fh_n \oplus \fut_n$, hence $\bigwedge^2\fb_n/\fh_n = \bigwedge^2i\fh_n \oplus i\fh_n\wedge \fut_n \oplus \fut_n \wedge \fut_n$. Moreover we have that $[\fh_n,i\fh_n] = 0$ and $[\fh_n,\fut_n]= \fut_n$.  
By derivation of the invariance condition:
\[
\forall a \in \fh_n, \forall X,Y \in \fut_n, \quad \phi([a,X],Y) + \phi(X,[a,Y])=0.
\]
From this equation one gets immediately that all forms in $i\fh_n\wedge i\fh_n$ are invariant, and by further close inspection,  that $\phi$ on $i\fh_n\wedge \fut_n$ is $0$. 

A direct and straightforward computation shows that on $\fut_n\wedge \fut_n$ the forms appearing in point (1) are the unique invariant $2$-forms on this space.

Linear independence is immediate by checking on suitable elements of $\fb_n/\fh_n$.
\end{proof}
 
%

As a corollary, the trivialization we are looking for is a linear combination of the forms in Lemma~\ref{lem:trivsurborel}. Let us find first a suitable candidate.
Given matrices $x, y\in\mathfrak b$, write them as $x=x_d+x_u$ and $y=y_d+y_u$   with $x_d, y_d\in\mathfrak{h}_n+i\mathfrak{h}_n$ diagonal and 
$x_u,y_u\in \mathfrak{ut}_n$ unipotent (strictly upper triangular). 
Define
\begin{equation}
\label{eqn:beta}
 \begin{array}{rcl}
   \beta\colon \mathfrak{b}_n\times \mathfrak{b}_n & \to & \mathbb R \\
   (x,y)\ & \mapsto & \Im \tr (x_u  {}^t\overline{y_u}- {}^t\overline {x_u} y_u)/4= {i} \tr ({}^t\overline {x_u} y_u-x_u {}^t\overline {y_u})/{4}.
 \end{array}
\end{equation}
For $(a_{kl}), (b_{kl}) \in\mathfrak{b}_n$ (i.e.~$a_{kl}=b_{kl}=0$ for $k>l$), \eqref{eqn:beta} is equivalent to:
$$
\beta((a_{kl}), (b_{kl})) = \frac{i}{4}\sum_{k<l} (\overline a_{kl} b_{kl}-a_{kl}\overline b_{kl})= \frac{1}{2} \sum_{k<l} \Im  (  a_{kl} \overline b_{kl} ),
$$
so
\[
\beta =  \frac{1}{2}\sum_{k<l} ur_{kl} \wedge ui_{kl}.
\]
In particular,  in this formula coefficients in the diagonal do not occur. A straightforward computation yields:

\begin{lemma} The following equality holds true: $\delta(\beta)=\varpi\vert_{\mathfrak{b}_n}$.
\end{lemma}

\begin{proposition}\label{prop:trivsurB}
The form $\beta$ above is the unique \emph{bounded} $2$-form  $\beta \in 
\Omega^2_{dR}(B/T)^B$ such that $d\beta = \varpi\vert_B$. It is characterized by the 
fact that it is the unique trivialization that is $0$ on the intersection $B \cap 
B^{-}$, where  $B^{-}$ is the opposite Borel subgroup of lower triangular matrices. 
\end{proposition}
\begin{proof} Since Lemma~\ref{lem:trivsurborel} gives a basis for $\Omega^2_{dR}(B/T)^B$,  any other invariant  trivialization of $\varpi$ restricted to $\fb_n$ differs from $\beta$  by a term of the form:
\[
 \sum_{s,r} \gamma_{sr}ih_s^\vee \wedge ih_r^\vee.
\]

To show that the coefficients $\gamma_{sr}$ are all $0$ observe that fixing a 
pair of indices $s,r$, the exponential of the elements $ih_s$, $ih_r$ give us a 
flat $\mathbb{R}^2 \subset B/T$. On this flat the volume form is trivial by 
direct inspection, and so are the forms $ur_{kl}^\vee \wedge ui_{kl}^\vee$ and 
$ih_p^\vee \wedge ih_q^\vee$  if $\{p,q\} \neq \{s,r\}$. So our invariant form 
on this flat restricts to the multiple $\gamma_{sr} ih_s^\vee \wedge ih_r^\vee$ 
of the euclidean volume form; this is bounded if and only if $\gamma_{sr}=0$.

So the unique candidate for a bounded trivialization  is $\beta$, and
since we know that there has to be one bounded trivialization, this is it.
\end{proof}

As a form in $\Omega^2_{dR}(B/T)^B$,  $\beta$ corresponds to the construction of Weinhard 
in \cite[Corollary~2.4]{MR2900175}, by means of a Poincar\'e lemma with respect an ideal point.

Summarizing, the class 
$\operatorname{vol}_{\mathcal{H}^3,\partial} \in 
\HH^3_c(\mathrm{SL}_n(\C);\{B_i\})$ is represented in the relative de Rham 
complex $\Omega_{dR}^*(\mathrm{SL}_n(\C)/B)^{ \mathrm{SL}_n(\C)}\oplus \bigoplus_i 
\Omega_{dR}^{\ast-1}(B_i/B_i \cap \mathrm{SU}(n))^{B_i}$, by a relative cocycle 
where:
\begin{enumerate}
\item The absolute part is given by the invariant $3$-from:
\[
\begin{array}{rcl}
\varpi : \bigwedge^3 \fsl_n : & \longrightarrow  & \mathbb{R} \\
(A,B,C) & \longmapsto & -2i\tr(pr_{i\fsu}A[pr_{i\fsu}B,pr_{i\fsu}C]).
\end{array}
\]
\item The relative part is given  on the copy $\Omega^2_{dR}(B_i/B_i\cap 
\mathrm{SU}(n))^{B_i}$ determined by the Borel subgroup $B_i$,  by choosing an 
arbitrary  element $h_i \in \mathrm{SU}(n)$ such that $h_i^{-1}Bh_i = B_i$, then 
and extending by invariance the $2$-form on $T_{\operatorname{Id}}(B_i/B_i \cap 
\mathrm{SU}(n))$ defined by $\beta_i= \operatorname{Ad}_{H}^\ast(\beta)$, where:
\[
\begin{array}{rcl}
  \beta\colon \fb_n\times \fb_n & \to  & \mathbb{R} \\
  (x,y) & \mapsto & {i} \tr ({}^t\overline {x_u} y_u-x_u {}^t\overline {y_u})/{4}.
 \end{array}
\]
here $x_u, y_u\in\mathfrak{ut}_n$ are the respective unipotent parts of $x$ and $y$.
\end{enumerate}
By construction the data $(\varpi,\{\beta_i\})$ forms a relative $2$-cocycle on 
$\fsl_n$ relative to the family of Borel Lie subalgebras $\{\fb_i\}$. 

%
%
%

\subsubsection{Volume and the Veronese embedding}

As an application let us show a formula relating the volume of a finite volume  hyperbolic $3$-manifold and the volume of its defining representation composed with the unique irreducible rank $n$ representation of $\mathrm{SL}_2(\C)$ induced by the Veronese embedding. This formula is proved  in   \cite[Proposition 21]{BBIarXiv14}, with different techniques (see also~\cite[Thm.~1.15]{GTZ}).
 
Let $\sigma_n\colon \operatorname{SL}_2(\C)\to\operatorname{SL}_{n}(\C) $, denote the $n$-dimensional irreducible representation. 
Namely $\sigma_n$ is the $(n-1)$-th symmetric product, induced by the Veronese embedding $\mathbb{CP}^1\to \mathbb{CP}^{n-1}$.

\begin{proposition}\cite{BBIarXiv14}\label{prop:volandVeronese}
\label{prop:volsigman}
 For $\rho\colon\pi_1(M)\to \operatorname{SL}_2(\C)$, 
 $\operatorname{vol}(\sigma_n\circ\rho)= {n+1 \choose 3}  \operatorname{vol}(\rho)$.
\end{proposition}

Recall that given \emph{any} family of Borel subgroups $\{B\}$, the map that forgets the relative part induces a natural isomorphism in continuous cohomology:
\[
\xymatrix{
\HH^3_c(\mathrm{SL}_n(\C),\{B\}) \ar[r] &  \HH^3_c(\mathrm{SL}_n(\C)).
}
\] 
Therefore to prove Proposition~\ref{prop:volandVeronese}, by the van Est isomorphism we only need to understand the effect of the induced map $\sigma_n: \fsl_2 \rightarrow \fsl_n$ on the absolute part $\varpi$ of the volume cocycle. Denote by $\varpi_n$ this absolute part, seen as a cocycle on $\fsl_n$, to emphasize its dependence on the index $n$.

\begin{lemma}\label{lem:effectonabspart}
Let $\sigma_n\colon \mathfrak{sl}_2\to  \mathfrak{sl}_n$ denote the representation of Lie
algebras induced by the irreducible representation that comes from the Veronese embedding.  Then:
$$
\sigma_n^*(\varpi_n)= {n+1 \choose 3} \varpi_2 .
$$
\end{lemma}

\begin{proof}
The result is a consequence of the fact that $\sigma_n(i\mathfrak{su}_2 )\subset i\mathfrak{su(n)}$ and  
 the equalities, for $a,b\in  \mathfrak{sl}_2(\C)  $:
 \begin{eqnarray*}
  [\sigma_n(a),\sigma_n(b)]&=&\sigma_n([a,b])\, , \\
  \operatorname{tr}(\sigma_n(a) \sigma_n(b)) &=& {n+1 \choose 3}   \operatorname{tr}(a\, b) \, .
 \end{eqnarray*}
The first equality is just a property of Lie algebra representations. For the second one, compute the image of a basis
of $ \mathfrak{sl}_2(\C)$:
$$
\sigma_n\begin{pmatrix}
         1 & 0 \\
         0 & -1
        \end{pmatrix}
 = 
\begin{pmatrix}
   n-1 &&& 0 \\
   & n-3 && \\
   &&\ddots & \\
   0 &&& 1-n
  \end{pmatrix} ,
  \ 
  \sigma_n\begin{pmatrix}
         0 & 1 \\
         0 & 0
        \end{pmatrix}
 = 
\begin{pmatrix}
   0 & n-1&&& 0 \\
   & 0 & n-2 &&\\
    &&0  &\ddots&   \\
   && & \ddots & 1 \\
   0 &&&& 0
  \end{pmatrix}
$$
$$
\textrm{ and }
\qquad 
\sigma_n\begin{pmatrix}
         0 & 0 \\
         1 & 0
        \end{pmatrix}
 = 
\begin{pmatrix}
   0 & &&& 0 \\
   1 & 0 &  &&\\
    &2& 0  &&   \\
   && \ddots & \ddots & \\
   0 &&& n-1 & 0
  \end{pmatrix} \, .
$$
By bilinearity, we just need to check the formula on the basis,
which is straightforward from the sums
  \begin{eqnarray*}
    (n-1)^2+(n-3)^2+\cdots +(1-n)^2 &=& 2 {n+1 \choose 3}\, ,\\
   (n-1) 1+(n-2)2+\cdots  + 1 (n-1) & = & {n+1 \choose 3}\, .
  \end{eqnarray*}
\end{proof}

\subsection{The volume as a characteristic class}\label{subsec volascharac}

In this section we recall briefly how a differentiable deformation of a representation 
translates into a differentiable deformation of a connection on the associated flat 
principal bundle. We will also recall  how integration on ${M}$ of 
pull-backs of invariant cocycles on $\mathrm{SL}_n(\C)$ by using a developing 
map give the interpretation of the volume form as a characteristic class. 
 
Recall that  $\pi=\pi_1(M)$ is the fundamental group of a compact manifold $M$ whose interior caries a hyperbolic metric  of finite volume. In particular the boundary   $\partial M$, if not empty, is a disjoint union of finitely many tori $T_1 \sqcup \cdots \sqcup T_k$. Since $\pi$ is a discrete group, associated to our fixed  representation $\rho: \pi \longrightarrow \mathrm{SL}_n(\C)$ there is a flat 
principal fibration:
\[
\xymatrix{
 \mathrm{SL}_n( \C )\ar@{^{(}->}[r] & E_\rho  \ar@{->>}[r] & M
}
\]
The total space $E_\rho$  is constructed as 
$$
E_\rho=\widetilde M\times \mathrm{SL}_n(\mathbb{C})/\pi
$$
where $\widetilde{M}$ is the universal covering space of $M$,  $\gamma\cdot 
(x,g)=(\gamma x,\rho(\gamma) g)$, for $\gamma\in\pi$, $x\in \widetilde M$, 
and $g\in \mathrm{SL}_n(\mathbb{C})$.
The natural flat connection
$$
\nabla\colon T E_\rho \to\mathfrak{sl}_n
$$
 is induced by the composition of the projection to the second factor of $T 
(\widetilde M\times \mathrm{SL}_n(\mathbb{C}))\cong T\widetilde M\times T 
\mathrm{SL}_n(\C)$ and the identification $T_g \mathrm{SL}_n(\C)\cong 
\mathfrak{sl}_n$ via $l_{g\ast}$
where $l_g$ denotes left multiplication by 
 $g\in \mathrm{SL}_n(\C)$.

Notice that $E_\rho$ is a non-compact manifold with boundary $\partial E_\rho$ 
that fibres over $\partial M$. Recall that in Subsection \ref{subsec:notvolrep} 
we have fixed a path from our base point in $M$ to a base point on each boundary 
component. Fix a base point on each covering space of each  boundary component 
$\partial M_i$, this induces commutative diagrams by sending the base point to 
the chosen  path to $\partial M_i$:
\[
\xymatrix{
\widetilde{\partial M_i} \ar@{^{(}->}[r] \ar[d] & \widetilde{M} \ar[d] \\
\partial M_i \ar@{^{(}->}[r] & M
}
\] 

 Since the restriction of $\rho$ to each parabolic subgroup $P_i \simeq 
\pi_1(\partial M_i)$ takes values in a Borel subgroup $B_i$, over the component 
$\partial M_i$, this restricted fibration 
\[
\xymatrix{
\operatorname{SL}_n(\C)\ar@{^{(}->}[r] & \partial E_\rho \ar@{->>}[r] & \partial 
M_i
}
\]
is obtained by extending the fibre from the flat fibration 
\[
\xymatrix{
B_i \ar@{^{(}->}[r] & \partial \widetilde{M_i} \times_{\rho} B_i \ar@{->>}[r] & \partial M_i
}
\]
along the inclusion $B_i \hookrightarrow \mathrm{SL}_n (\C)$. In particular the flat connection $\nabla$ restricted to a component $\partial M_i$ takes values in the lie algebra $\fb_i$ of the chosen Borel $B_i$.

As $M$ is aspherical,  $\dim M\leq 3$, and $\mathrm{SL}_n(\C)$ is $2$-connected, by Whitehead's theorem there exists a $\rho$-equivariant map that sends the base point in $\widetilde{M}$ to $\operatorname{Id}$:
\[
D\colon\widetilde M\to \mathrm{SL}_n(\C)\, .
\]

By precomposing this map with our fixed inclusions of the universal covering spaces of the boundary components, we get for each of those a compatible developing map:
\[
\xymatrix{
 \widetilde{M} \ar[r]^D & \mathrm{SL}_n(\C) \\
 \widetilde{\partial M_i} \ar@{^{(}->}[u] \ar[r]_{D_i} & B_i \ar@{^{(}->}[u]
}
\]

On the one hand the developing map induces a trivialization $\Theta_\rho$ of the flat bundle or equivalently, a section $s_\rho$ to the fibration map:
$$
\begin{tikzpicture}
  \matrix (m) [matrix of math nodes,row sep=3em,column sep=3em,minimum width=2em]
  {
\widetilde{M} \times \mathrm{SL}_n(\C) & \widetilde{M} \times \mathrm{SL}_n(\C) \\[-35pt]
(x,g)  &(x,D(x)g) \\
     M\times \mathrm{SL}_n(\C)  &E_{\rho} \\};
  \path[-stealth]
  (m-1-1) edge  (m-1-2)
 (m-2-1)  edge [|->] (m-2-2)
   (m-3-1) edge  node [above] {$\Theta_\rho$} (m-3-2)
        (m-2-1) edge (m-3-1)
        (m-2-2) edge (m-3-2);
\end{tikzpicture}
\qquad 
\begin{tikzpicture}
  \matrix (m) [matrix of math nodes,row sep=3em,column sep=3em,minimum width=2em]
  {
\widetilde{M}  & \widetilde{M} \times \mathrm{SL}_n(\C) \\[-35pt]
x  &(x,D(x)) \\
     M  &E_{\rho} \\};
  \path[-stealth]
  (m-1-1) edge  (m-1-2)
 (m-2-1)  edge [|->] (m-2-2)
   (m-3-1) edge  node [above] {$s_\rho$} (m-3-2)
        (m-2-1) edge (m-3-1)
        (m-2-2) edge (m-3-2);
\end{tikzpicture}
$$
Both maps are of  course related:
$$
s_\rho=\Theta_\rho \circ s\, ,
$$
where $s\colon M\to M\times \mathrm{SL}_n(\C)$ is the constant section of the trivial bundle, given by fixing $\operatorname{Id} \in \mathrm{SL}_n(\C)$ as second coordinate. 
The composition of the section with the flat connection
$$
\nabla\circ (s_\rho)_\ast\colon TM\to\mathfrak{sl}_n
$$
is used to evaluate characteristic classes of $ \mathfrak{sl}_n $.

The trivialization $\Theta_\rho$ is used to pull back the connection on $E_\rho$ to the trivial bundle:
$$
\nabla_\rho\overset{\textrm{def}}=\nabla\circ (\Theta_\rho)_\ast\colon T (M\times \mathrm{SL}_n \C)\to \mathfrak{sl}_n
$$
In this way, when we deform $\rho$, we deform $\nabla_\rho$ on the trivial bundle, because
$$
\nabla_\rho\circ s_*= \nabla\circ (s_\rho)_*\, .
$$
On the other hand, the developing map models the map induced in continuous cohomology by the representation $\rho$ in the following way. 
Recall from~\cite[Prop. 5.4 and Cor 5.6]{MR554917} that if $N$ is a smooth manifold on which $G$ acts properly smoothly then the complex $\Omega_{dR}^\ast(N)^G$ computes the continuous cohomology of $G$. Moreover the map in continuous cohomology  induced by a continuous homomorphism $\rho: G  \rightarrow H$ can be computed  by considering a $\rho$-equivariant map $R: N \rightarrow M$ where $N$ is a $G$-manifold as above and $M$ an $H$-manifold. 
By definition this is exactly what the developing map $D$ 
is with respect to the continuous map $\rho: \pi \rightarrow \mathrm{SL}_n(\C)$. Indeed, by the above  cited result, we have  the known fact that the canonical inclusion  $\Omega_{dR}^\ast(\widetilde{M})^{\pi_1(M) }\rightarrow \Omega_{dR}^\ast(M)$ is a quasi-isomorphism.

 The same discussion holds true for each boundary component since each of these is a $K(\mathbb{Z}^2,1)$, and the compatibility of the developing maps on $M$ and on its boundary components imply that they induce via the cone construction the map:
\[
\rho\ast: \HH^\ast_c(\mathrm{SL}_n(\C),\{B\}) \rightarrow \HH^3(M,\partial M).
\]

Let us be slightly more precise and let us revisit  our previous Definition~\ref{def:volrepborl} of the volume. At the level of de Rahm cochains, the volume class $\operatorname{vol}_{\mathcal{H},\partial}$ is represented by the relative cocycle $(\varpi,{\beta})$ constructed in Section~\ref{subsec:relcoc}. Since evaluation on the fundamental class translates in de Rahm cohomology into integrating, by Stoke's formula and the above discussion:

\begin{definition}\label{def:defvolplusprecis}
	Let $\rho\colon \pi \rightarrow \mathrm{SL}_n(\C)$ be a representation of the 
	fundamental  a $3$-manifold $M$ whose interior is an hyperbolic manifold of finite volume. Denote the boundary components of $M$ by
	$ T_1\sqcup \cdots\sqcup T_k$. Fix a system of peripheral 
	subgroups $P_i$ in $\pi$ and for each such group fix a Borel subgroup $B_i 
	\subset \mathrm{SL}_n(\C)$ such that $\rho(P_i) \subset B_i$. Denote by $D$ the developing map associated to $\rho$ and by $D_r$ its restriction to the universal cover of the boundary component $T_r$.   Then
	\begin{equation}
	\label{eqn:Stokes}
	\operatorname{Vol}(\rho) 
	=  \int_{M}^{}D^\ast(\varpi) - \sum_{r=1}^k\int_{T_r}^{}D_r^\ast(\beta_r).
	\end{equation}
	Where the differential forms $D^\ast(\varpi)$ and $D_r^\ast(\beta_r)$ descend from the universal covers to differential forms on the manifolds by equivariance. 
\end{definition}

Now, since $\mathrm{SL}_n(\C)$ is $2$-connected, the Leray-Serre spectral sequence  in relative cohomology gives us a short exact sequence
\[
\xymatrix{
0 \ar[r] & \HH^3(M,\partial M) \ar[r] & \HH^3(E_\rho,\partial E_\rho) \ar[r] & \HH^3(\mathrm{SL}_n \C) \ar[r] & 0
}
\]

In particular, the volume class $\rho^\ast(\operatorname{vol}_{\mathcal{H},\partial})$ defined in Section~\ref{sec:volrelatif} can be seen as a class in $\HH^3(E_\rho,\partial E_\rho$). The key observation of Reznikov in~\cite{MR1412681} is that in  this larger group the volume class can be interpreted as a characteristic class associated to  the foliation of $E_\rho$ induced by the flat connection.

\begin{proposition}\label{prop:volischarclas}
Denote by $j^*:\HH^3(M,\partial M) \rightarrow \HH^3(E_\rho,\partial E_\rho)$ the morphism induced by the projection $E_\rho \rightarrow M$ in de Rham cohomology. Then
\[
j^\ast(\rho^\ast(\operatorname{vol}_{\mathcal{H},\partial})) = \operatorname{Char}_{\nabla_\rho,{\nabla\vert_{\partial M_i}}}(\varpi,\{\beta_i\}).
\]
\end{proposition}
\begin{proof}
First observe that $\operatorname{Char}_{\nabla_\rho,{\nabla\vert_{\partial M_i}}}(\varpi,\{\beta_i\}) \in \ker( \HH^3(E_\rho,\partial E_\rho) \rightarrow \HH^3(\mathrm{SL}_n(\C))= \textrm{Im}j^\ast$. Indeed, by construction, the restriction of this characteristic class to the fibre $\mathrm{SL}_n(\C)$ is given by the form $\varpi$. But the inclusion $\mathrm{SU}(n) \rightarrow \mathrm{SL}_n(C)$ is a weak equivalence, hence induces an isomorphism in cohomology, and since $\omega$ only depends on the projection on $i\fsu_n$, the restriction of $\varpi$ to $\mathrm{SU}(n)$ is the trivial form. So to check the equality we only need to show that after composing with the map induced by the section $s_\rho$ both sides of the equation agree.  Recall that by construction $(\omega, \{\beta_i\})$ is a relative cocycle that represents the hyperbolic form in $\HH^3_c(\mathrm{SL}_n(\C),\{B\})$. Hence by the discussion on the map $D$ the class $\rho^\ast(\operatorname{vol}_{\mathcal{H},\partial}) $ is represented by the cocycle $D^\ast((\omega,\{\beta_i\}))$.

To finish the proof it is enough to show  that we have a commutative diagram: 
\[
\xymatrix{
C^3(\fsl_n,\{\fb\}) \ar[r]^{\operatorname{Char}_{\nabla_\rho,{\nabla\vert_{\partial M_i}}}} \ar[d] & \Omega^3_{dR}(E_\rho,\partial E_\rho) \ar[r]^-{\Theta_\rho^\ast} \ar[dr]_{s_\rho^\ast}&   \Omega^3_{dR}( M \times \mathrm{SL}_n(\C), \{\partial M_i \times B_i\}) \ar[d]^ {s^\ast} \\ 
\Omega^3_{dR}(\mathrm{SL}_n(\C)/\mathrm{SU}(n), \{B_i/T_i\}) \ar[rr] & & \Omega^3_{dR}(M,\partial M) 
}
\]
where the bottom row is induced by $D$ and the quasi-isomorphisms $\Omega_{dR}^\ast(\widetilde{M}) ^\pi \rightarrow \Omega_{dR}^\ast(M)$ and $\Omega_{dR}^\ast(\widetilde{\partial M_i})^{\pi_1(\partial M_i)} \rightarrow \Omega_{dR}^\ast(\partial M_i)$.

As this is a  diagram on the chain level in relative cohomology, it is enough to  check that the corresponding absolute maps yield commutative diagrams and are compatible, i.e. for the absolute part, 

\[
\xymatrix{
C^3(\fsl_n) \ar[r]^{\operatorname{Char}_{\nabla_\rho}} \ar[d] & \Omega^3_{dR}(E_\rho) \ar[r]^-{\Theta_\rho^\ast} \ar[dr]_{s_\rho^\ast}&   \Omega^3_{dR}( M \times \mathrm{SL}_n(\C)) \ar[d]^ {s^\ast} \\ 
\Omega^3_{dR}(\mathrm{SL}_n(\C)/\mathrm{SU}(n)) \ar[r]^-{D^\ast} &  \Omega^3_{dR}(\widetilde{M})^{\pi} \ar[r] & \Omega_{dR}^3(M)
}
\]
and analogously for the relative part:
\[
\xymatrix{
C^2(\fb_i) \ar[r]^{\operatorname{Char}_{\nabla\vert_{\partial M_i}}} \ar[d] & \Omega^2_{dR}(\partial E_\rho) \ar[r]^-{\Theta_\rho^\ast} \ar[dr]_{s_\rho^\ast}&   \Omega^2_{dR}( \partial M_i \times B) \ar[d]^ {s^\ast} \\ 
\Omega^2_{dR}(B_i/T_i) \ar[r]^-{D^\ast} &  \Omega^2_{dR}(\widetilde{\partial M})^{\pi_1(\partial{M})} \ar[r] & \Omega_{dR}^2(\partial M)
}
\]
Both the proof of commutativity of the diagrams  and the compatibility are now elementary diagram chases.
\end{proof}

\section{The variation formula}\label{sec:varformula}

We are now ready to collect our efforts; but first a word of caution on the 
smoothness of the variety of representations.
The algebraic variety $\mathrm{Hom}(\pi_1M,\mathrm{SL}_n(\C)$ is not diferentiable in 
general, in fact for $M$ compact the singularities that 
appear can be as wild as possible, for a discussion of the singularities see for 
instance~\cite{MR3654101}. Nevertheless, by Whitney's theorem 
the algebraic variety $\mathrm{Hom}(\pi_1 M,\mathrm{SL}_n(\C))$ is generically 
smooth (i.e. the non-smooth locus is of Lebesgue mesure zero). Even restricted 
to the smooth locus, the volume function itself is not everywhere differentiable as is 
transparent from previous work of Neumann-Zagier. More precisely let us check: 

\begin{lemma}\cite{NZ}\label{lem:volnondiff}
 	For $n=2$ and a manifold with a single boundary component, the volume function is not differentiable at the defining representation.
\end{lemma}

Recall that the defining representation is the one corresponding to the complete 
hyperbolic structure on the interior of $M$. In \cite{NZ}, Neumann and Zagier 
use a parameter $u \in \C$ in a neighbohood of the origin to parametrize a 
neighborhood of the complete structure in the moduli space of hyperbolic ideal 
triangulations. As noticed in their work~\cite{NZ}, $u$ and $-u$ correspond to 
se same hyperbolic metric on the interior of $M$. In fact $\mathrm{Hom}(\pi_1 
M,\mathrm{SL}_2(\C))/SL_2(\C)$ is locally parametrized by
\[
\operatorname{trace}(\rho_u(l)) = \pm 2\cosh(u/2) = \pm (2 + u^2/4 + O(u^4)),
\]
where $\rho_u$ denotes the holonomy of the structure with parameter $u$. In 
particular $\rho_0$ is the defining representation. Then Neumann-Zagier define 
an analytic function $v(u)$ such that $\operatorname{trace}(\rho_u(m))= \pm 
2\cosh(v/2)$ and prove that $v= \tau u + O(u^3)$, where $\tau \in \C$ is the so 
called cusp length with $\Im(\tau)>0$, and
\[
\operatorname{vol}(\rho_u) = \operatorname{vol}(\rho_0) + 
\frac{1}{4}\Im(u\overline{v}) + O(|u|^4) = \operatorname{vol}(\rho_0) + 
\frac{1}{4}\Im(\tau)|u|^2 + O(|u|^4). 
\]
Thus, by choosing a local parameter $z=2\cosh(u/2) -2 = u^2/4 + O(u^4)$ in a 
neighborhood of the origin, the volume function has an expansion of the form:
\[
\operatorname{vol}(\rho_u)-\operatorname{vol}(\rho_0) = - \Im(\tau)|z| + O(|z|^2).
\]
Hence the volume is not a differentiable function on $\mathrm{Hom}(\pi_1 
M,\mathrm{SL}_2(\C))/\mathrm{SL}_2(\C)$, as $z \mapsto |z|$ is not 
differentiable at $z=0$. Neither is the volume differentiable on the variety of 
representations, because the projection $\mathrm{Hom}(\pi_1 M,\mathrm{SL}_2(\C)) 
\rightarrow \mathrm{Hom}(\pi_1 M,\mathrm{SL}_2(\C))/\mathrm{SL}_2(\C)$ is a 
fibration in a neighborhood of $\rho_0$.

This being said, let us go back now to our variation formula.
The following two subsections conclude the proof of the main theorem. 

\subsection{The variation comes from the boundary}
 
Recall that  the group $\mathrm{SU}(n)$ 
acts transitively by conjugation on the set of Borel subgroups of 
$\mathrm{SL}_n(\C)$.  Then by uniqueness of the trivialization $\beta$ proved in 
Proposition~\ref{prop:trivsurB}, the trivializations of the volume form on two 
different Borel subgroups $B_1$ and $B_2$, say  $\beta_1$ and $\beta_2$, are 
compatible in the sense of that if  $H \in \mathrm{SU}(n)$ is chosen such that 
$HB_2H^{-1} =B_1$, then for any $b,b' \in \fb_2$, $\beta_2(b,b') = 
\beta_1(\operatorname{Ad}_{H}b,\operatorname{Ad}_{H}b')= 
\operatorname{Ad}_H^\ast(\beta_1)$.

Let $\rho_t: \pi_1(M) \rightarrow \mathrm{SL}_n(\C)$ be a differentiable family of 
representations. As we discussed before, we may think of the associated flat 
bundles $E_{\rho_t}$  as being the flat bundle $E_{\rho_0}$ but with a varying 
family of connections $\nabla_t$. The uniqueness property discussed in the 
previous paragraph is  precisely the coherence requirement of 
Definition~\ref{def coherencecond} with respect to the subgroup $H=\mathrm{SU}(n)$. 
Consistent with our conventions at the end of  Section~\ref{subsec varchaclass}, we will decorate with a subscript as in 
$\HH^\ast_{\mathrm{SU}(n)}$ the cohomology of the complexes  
$C^\ast_{\mathrm{SU}(n)}(\mathfrak{g};\mathbb{R})$, etc. defined in 
Notation~\ref{not:relcochcoplx} and Definition~\ref{def relcohcoxes}
 of Section~\ref{subsec varchaclass}.

We can now apply the results of Section~\ref{subsec volascharac} to compute the variation of the volume. By the construction of the factorization of the variation map:
\[
\xymatrix{
\HH^\ast_{\mathrm{SU}(n)}(\fsu_n,\{\fb\}) \ar[r]^-{\operatorname{var}} \ar@/_{1pc}/[drr]_{\operatorname{Var}}&  \HH^{\ast-1}_{\mathrm{SU}(n)}(\fsu_n,\{\fb\}; \fsu^\vee_n,\{\fb^\vee\})\ar[r]^-{\Fu}  &  \HH^\ast_{\mathrm{SU}(n)}(\widetilde{\fsu}_n,\{\widetilde{\fsu}_n\})\ar[d]^{\operatorname{Char}_{\nabla_t,\{\widetilde{\nabla}_t\}}} \\
& &  \HH^\ast_{dR}(M,\partial M)
}
\]
we have a commutative diagram
\[
\xymatrix{
\HH^2_{\mathrm{SU}(n)}(\fsl_n;\mathbb{R}) \ar[r] \ar[d]^{\operatorname{var}} & \prod \HH^{2}(\fb;\mathbb{R}) \ar[r] \ar[d]^{\operatorname{var}} & \HH^3_{\mathrm{SU}(n)}(\fsl_n, \{\fb\};\mathbb{R}) \ar[r] \ar[d]^{\operatorname{var}} & \HH^3_{\mathrm{SU}(n)}(\fsl_n;\mathbb{R})  \ar[d]^{\operatorname{var}} & \\
\HH^1_{\mathrm{SU}(n)}(\fsl_n;\fsl_n^\vee) \ar[r] \ar[d] & \prod \HH^{1}(\fb;\fb^\vee) \ar[r] \ar[d] & \HH^2_{\mathrm{SU}(n)}(\fsl_n, \{\fb\};\fsl_n^\vee,\{\fb^\vee\}) \ar[r] \ar[d] & \HH^2_{\mathrm{SU}(n)}(\fsl_n;\fsl_n^\vee) \ar[d] &   \\
\HH^2(M;\mathbb{R}) \ar[r] & \prod \HH^{2}(\partial{M}) \ar[r] & \HH^3(M,\partial M) \ar[r] & \HH^ 3(M) \simeq 0 & 
}
\]

Let us recall the following lemma by Cartier~\cite[Lemme 1]{cartier}:

\begin{lemma}\label{lem:cartsemsimple}
Let $V$ be a vector space on a field $k$ and $A$ be a family of endomorphisms of $V$. Assume that $V$ is completely reducible. Denote by $V^\sharp$ the subspace of those vectors annihilated by all the $X \in A$. and by $V^0$ the subspace generated by the vectors $Xv$ ($X \in A, v \in V$).
\begin{enumerate}
\item $V = V^\sharp \oplus V^0.$
\item If $V$ is equipped with a differential $d$ that commutes to the $X \in A$, and such that $Xv$ is a boundary if $v$ is a cycle, then the homology with respect to this boundary gives $H(V) \simeq H(V^\sharp).$
\end{enumerate}
\end{lemma}

\begin{corollary}\label{cor annulcohoslsldual}
For $\ast \geq 1$ and any $n \geq 2$, 
\[
\HH^\ast(\fsl_n;\fsl_n^\vee) \simeq 0 \simeq \HH^\ast_{\mathrm{SU}(n)}(\fsl_n;\fsl_n^\vee).
\]
\end{corollary}
\begin{proof}
That $\HH^\ast(\fsl_n;\fsl_n^\vee) \simeq 0$ is the direct application that 
Cartier makes of his lemma, given that $\fsu_n$ is semi-simple.

For the second isomorphism, we apply Lemma~\ref{lem:cartsemsimple} to the (acyclic!) complex $V= C^\ast(\fsl_n;\fsl_n^\vee)$ viewed as a (graded) vector space acted upon by $SU(n)$. Since $\mathrm{SU}(n)$ is compact then $V$ is indeed completely reducible. 
Moreover, by functoriality of the complex, its differential commutes with the 
action of the elements in $\mathrm{SU}(n)-\operatorname{Id}$. If $v$ is a cycle, 
and $X \in \mathrm{SU}(n)-\operatorname{Id}$, then $Xv$ is a cycle, and by 
acyclicity of this complex, it is a boundary. Observe that being annihilated by 
$A-\operatorname{Id} $ is the same as being fixed by $A$, hence 
Lemma~\ref{lem:cartsemsimple} tells us that the embedding 
$C^\ast_{\mathrm{SU}(n)}(\fsl_n;\fsl_n^\vee) \hookrightarrow 
C^\ast(\fsl_n;\fsl_n^\vee)$ is a quasi-isomorphism.
\end{proof}

As a consequence our diagram above boils down to:
\[
\xymatrix{
 & \prod \HH^{2}(\fb;\mathbb{R}) \ar[r] \ar[d]^{\operatorname{var}} & 
\HH^3_{\mathrm{SU}(n)}(\fsl_n, \{\fb\};\mathbb{R}) \ar[r] \ar[d]^{\operatorname{var}} & 
\HH^3_{\mathrm{SU}(n)}(\fsl_n;\mathbb{R})  \ar[d]^{\operatorname{var}}  \\
0 \ar[d] \ar[r] & \prod \HH^{1}(\fb;\fb^\vee) \ar[r] \ar[d] & 
\HH^2_{\mathrm{SU}(n)}(\fsl_n, \{\fb\};\fsl_n^\vee,\{\fb^\vee\}) \ar[r] \ar[d] & 
0 \ar[d]    \\
\HH^2(M;\mathbb{R}) \ar[r] & \prod \HH^{2}(\partial{M}) \ar[r] & 
\HH^3(M,\partial M) \ar[r] & 0 
}
\]

In particular, the variation of the volume class is the image of a cohomology 
class in $\prod \HH^2(\partial M)$. To see which one we have to find an inverse 
to the isomorphism:
\[
\xymatrix{
\prod \HH^{1}(\fb;\fb^\vee) \ar[r]  & \HH^2_{\mathrm{SU}(n)}(\fsl_n, 
\{\fb\};\fsl_n^\vee,\{\fb^\vee\}).
}
\]
Unraveling the definitions it is given by the following construction.
The  map
\[
\xymatrix{
 C^2_{\mathrm{SU}(n)}(\fsl_n, \{\fb\};\fsl_n^\vee,\{\fb^\vee\}) \rightarrow C^2_{\mathrm{SU}(n)}(\fsl_n, ;\fsl_n^\vee).
}
\]
simply forgets the relative part, and acyclicity on the right hand side means 
that the absolute part $\operatorname{var}(\varpi)$ of the relative cocycle 
$\operatorname{var}(\varpi,\{\beta\})$ is 
a coboundary, say $\operatorname{var}(\alpha) = d\gamma$.
Then the preimage of $\operatorname{var}(\varpi,\{\beta\})$ in $\prod \HH^{1}(\fb;\fb^\vee)$ is 
given by the class of the family $\operatorname{var}(\beta) - i^\ast \gamma$, where $i^\ast$ is the 
map induced by the inclusion $\fb \rightarrow \fsu$.

\begin{lemma}\label{lem:trivvarvol}
The image of $\varpi$, the absolute part of the volume cocycle, under the map $\operatorname{var}\colon  C^3_{\mathrm{SU}(n)}(\fsl_n;\mathbb{R}) \rightarrow C^2(\fsl_n;\fsl_n^\vee)$  is  the coboundary of the cochain:
\[
\begin{array}{rcl}
\gamma\colon \fsl_n & \rightarrow &\fsl_n^\vee \\
g & \mapsto & h \leadsto  i\tr(pr_{i\fsu_n}(g)pr_{\fsu_n}(h)).
\end{array}
\]
where $pr_{\fsu_n}\colon  \fsl_n \rightarrow \fsu_n$ and $pr_{i\fsu_n} \colon  \fsl_n \rightarrow i\fsu_n$ are the canonical projections associated to the orthogonal decomposition $\fsl_n = \fsu_n \oplus i \fsu_n$.
\end{lemma}
\begin{proof}
For $x_1,x_2\in \mathfrak{sl}_n $, 
$$ d (\gamma) (x_1, x_2)= x_1 \gamma(x_2)- x_2\gamma(x_1)-\gamma([x_1,x_2]).
 $$
Recall that for $\theta\in \mathfrak{g}^\vee$ and $x,y\in\mathfrak{g}$, we have $(x\theta)(y)=-\theta([x,y])$.
Hence, for  $x_1,x_2, x_3\in \mathfrak{sl}_n $, 
\begin{multline*}
  d (\gamma) (x_1, x_2) (x_3)= -\gamma(x_2)([x_1,x_3 ])+ \gamma(x_1 )([x_2,x_3 ])- \gamma([x_1,x_2])(x_3)\\
       =  i \tr \big(
	    -pr_{i\fsu_n}(x_2) pr_{\fsu_n}([x_1,x_3]) + pr_{i\fsu_n} (x_1) pr_{\fsu_n} ([x_2,x_3])- pr_{i\fsu_n} ([x_1,x_2])pr_{\fsu_n} (x_3) 
       \big)
\end{multline*}
Since $[\fsu_n,\fsu_n ]\subset \fsu_n$,
$[ i\fsu_n, i\fsu_n]\subset \fsu_n$, and 
$[i\fsu_n, \fsu_n ]\subset i\fsu_n$,
\begin{eqnarray*}
 \delta (\gamma) (x_1, x_2) (x_3)&=
       &  i \tr \big(  -pr_{i\fsu_n} (x_2) ([ pr_{\fsu_n} (x_1), pr_{\fsu_n} (x_3)]+[ pr_{i\fsu_n} (x_1,) pr_{i\fsu_n} (x_3)]) \\
      && \phantom{i\tr \big( }  +pr_{i\fsu_n} (x_1) ([ pr_{\fsu_n} (x_2), pr_{\fsu_n} (x_3)]+[ pr_{i\fsu_n} (x_2), pr_{i\fsu_n} (x_3)]) \\
      && \phantom{i\tr \big( } - ([pr_{\fsu_n} (x_1),pr_{i\fsu_n} (x_2)] -[  pr_{i\fsu_n} (x_1),pr_{\fsu_n} (x_2)]  )pr_{\fsu_n} (x_3) 
       \big)\\
      && = 2 i\tr ( pr_{i\fsu_n} (x_1)[pr_{i\fsu_n} (x_2),pr_{i\fsu_n} (x_3) ] ).
\end{eqnarray*}
Here we have used that $(A,B,C)\mapsto \tr (A[B,C])$ is alternating.
 \end{proof}

Each Borel Lie algebra $\fb_n$ fits into a split exact sequence of Lie algebras:

\[
\xymatrix{
0 \ar[r] & \mathfrak{ut}_n \ar[r] & \fb_n \ar[r] & \ft_n \ar[r] & 0.
}
\]

We have a splitting $\ft_n = \fh_n \oplus i\fh_n$. Denote by $pr_{\fh_n}$ (resp. $pr_{i\fh_n}$) the projection onto $\fh_n$ (resp $i\fh_n$).
\begin{proposition}\label{prop:cocyclevariation}

The variation of the volume of a representation is given by the sum over of the integral over each boundary component of $\partial {M} $ of the image of the cohomology class of the $1$-cocycle in $C^1(\fb_n;\fb_n^\vee)
$:
\[
\begin{array}{rcl}
\zeta\colon  \fb_n & \longrightarrow & \fb_n^\vee \\
x & \longmapsto & y \leadsto   i\tr(pr_{i\fh_n}(x)pr_{\fh_n}(y))
\end{array}
\]
under the map 
\[
\HH^1(\fb_n; \fb_n) \rightarrow \HH^2(\partial {M}).
\]
\end{proposition}
\begin{proof}
As $\operatorname{var}(\varpi)$ is the coboundary of $\gamma$,
the cocycle $(\operatorname{var}(\varpi),\{\operatorname{var}(\beta_r)\}) $ is cohomologous to  $(0,\{\operatorname{var}(\beta_r)-i^*(\gamma) \})$. 
Therefore, as the integral on the  boundary $\partial M$ appears subtracting in Definition~\ref{def:defvolplusprecis}, the variation of volume is:
\begin{equation*}
 -\sum_{r=1}^k\int_{T_r}  ( s_\rho^\ast\circ\operatorname{Char}_{\nabla_t}\circ \Fu )    (\operatorname{var}(\beta_r)- i^*(\gamma) ) .
\end{equation*}
Hence we need to prove that $\zeta=i^*(\gamma)- \mathrm{var}(\beta)  $.
 Given  $x, y\in\mathfrak{b}_n$, write 
 $$x=x_d+x_u\qquad\textrm{ and }\quad y=y_d+y_u$$ with $x_u,\, y_u\in \mathfrak{u}_n$ and $x_d,y_d\in\mathfrak{h}+i\mathfrak{h}$ diagonal, their Chevalley-Jordan decomposition.
 Notice that $pr_{i\fsu_n}(x_d)=pr_{i\fh}(x)$ and $pr_{\fsu_n}(y_d)=pr_{\fh}(y)$ are diagonal, hence their product with elements of $\fut_n$ and 
  ${}^t\fut_n$ have trace zero. 
  Therefore:
  $$
  \gamma(x)(y)=i\tr(pr_{i\fh_n}(x)pr_{\fh_n}(y))+\gamma(x_u)(y_u)= \zeta(x)(y)+ \gamma(x_u)(y_u). 
  $$
  As $pr_{i\fsu}(x_u)=(x_u+{}^t\overline{x_u})/2$,  $pr_{\fsu}(y_u)=(y_u-{}^t\overline{y_u})/2$, and the trace vanishes on  $\mathfrak{u}_n$,
  $$
  \gamma(x_u)(y_u)= i \tr((x_u+{}^t\overline{x_u})(y_u-{}^t\overline{y_u}))/4= i \tr({}^t\overline{x_u}  y_u-x_u{}^t\overline{y_u})/4=\beta(x,y).
  $$
  Hence $i^*(\gamma)=\zeta+\operatorname{var}(\beta)$ as claimed.
\end{proof}

Observe that this form we have to integrate does only depend on the projection on $\fb_n/\mathfrak{ut}_n$. Recall that corresponding to the above split exact sequence of $\fb_n$ we have a split short exact sequence of Lie groups:
\[
\xymatrix{
1 \ar[r] & U_n \ar[r] & B_n \ar[r] & T_n\ar[r] & 1,
}
\]
where $U_n$ stands for the unipotent matrices, and the sequence is split by the semi-simple matrices in $B_n$.
Then the fact that the cochain $\zeta$ only depends on the projection onto 
$\ft_n$  means precisely that the variation of the volume depends on the 
restriction of the representation $\rho\colon P_i \rightarrow B_i$ only through its 
projection on $B_i/U_n$, a representation with values in an abelian group.

As an immediate corollary we have that if for each peripheral subgroup the restriction  of the representation $\rho$ take values in unipotent subgroups of $\mathrm{SL}_n(\C)$ and the deformation of $\rho$ is also boundary unipotent then the volume does not vary:

\begin{corollary}\label{cor:uniprepvolrigide}
The volume function restricted to the subspace of boundary unipotent representations is locally constant.
\end{corollary}

%
%
%

We now turn to a more explicit formula for the variation of the volume as encoded on each torus.

\subsection{Deforming representations on the torus}

Let $\{\alpha,\beta\}$ be a generating set of the fundamental group of the 
2-torus $T^2=\mathbb{R}^2/\mathbb{Z}^2$. They act on the universal covering 
$\alpha,\beta\colon  \mathbb{R}^2\to  \mathbb{R}^2$ as 
the integer lattice of translations:
\(\alpha(x,y)=(x+1,y)$ and $\beta(x,y)=(x,y+1)\). 

By the Lie-Kolchin theorem, the image  $\rho(\pi_1(T^2))$ is contained in a Borel subgroup $B_n$ and up to conjugation we assume that its variation is contained in a fixed subgroup.
The class we want to evaluate vanishes in $\mathfrak{u}_n$, so we do not need to understand the whole perturbation of $\rho$ in $B_n$ but just its projection to
$\pi\colon B_n\to  B_n/U_n= \Delta_n \cong \mathbf (\C^*)^{n-1}$.
Write 
$$
\pi(\rho(\alpha))=\exp(a),\ \pi(\rho(\beta))=\exp(b) \in \Delta_n,
$$
where $a,b\in\mathfrak{sl}_n(\C)$ are diagonal matrices. Notice that there is an indeterminacy of the logarithm, the nontrivial entries (diagonal) of $a$ and $b$  are only well defined
up addition of a term in to $2\pi i \,\Z $, but this does not affect the final result.

Since $\Delta_n$ is abelian, for such a representation we have a $\rho$-equivariant map
$$
\begin{array}{rcl}
  D\colon \mathbb{R}^2 & \to & B_n/U_n \\
    (x,y) & \mapsto & \exp(x\, a+ y\, b).
\end{array}
$$
Then
$$
(\nabla \circ (s_\rho)_\ast )\left(\tfrac{\partial\phantom{x}}{\partial x}\right)= a \quad \textrm{ and }\quad 
(\nabla \circ (s_\rho)_\ast )\left(\tfrac{\partial\phantom{x}}{\partial y}\right)= b\, .
$$
We vary the representation by varying $a$ and $b$, so
$$
(\dot \nabla \circ (s_\rho)_\ast )\left(\tfrac{\partial\phantom{x}}{\partial x}\right)= \dot a \quad \textrm{ and }\quad
(\dot\nabla \circ (s_\rho)_\ast )\left(\tfrac{\partial\phantom{x}}{\partial y}\right)= \dot b\, .
$$

\begin{lemma} \label{lemma:evaltorus}
For $c\in C^1(\mathfrak{g},\mathfrak{g}^\vee)$ and a variation as above, 
$$
\int_{\partial M}  (s_\rho)^\ast(\operatorname{Char}_{\nabla_t}(\Fu (c)) )=  c(a)(\dot b)-c(b)(\dot a)\, .
$$
\end{lemma}
\begin{proof}
For $Z_1, Z_2$ vector fields on $E_\rho\vert_{\partial M}$,
$$
\operatorname{Char}_{\nabla_t}(\Fu (c))(Z_1,Z_2)= c(\nabla (Z_1))(\dot \nabla(Z_2))-c(\nabla (Z_2))(\dot \nabla(Z_1))
$$
Setting $Z_1=(s_\rho)_\ast \left(\frac{\partial\phantom{x}}{\partial x}\right) $ and 
 $Z_2=(s_\rho)_\ast \left(\frac{\partial\phantom{y}}{\partial y}\right) $,
 $\nabla (Z_1)=a$, $\dot \nabla (Z_1)=\dot a$, $\nabla (Z_2)=b$, $\dot\nabla (Z_2)=\dot b$,
 hence
 $$
  (s_\rho)^\ast(\operatorname{Char}_{\nabla_t}(\Fu (c)) )=  (c(a)(\dot b)-c(b)(\dot a)) d\, x\wedge d\, y
 $$
 As $\int_{\partial M} d\, x\wedge d\, y=1$, the lemma follows.
 \end{proof}

\begin{corollary}
\label{corollary:formula}

If $a,b,\dot a, \dot b\in\mathfrak{b}_n$, then the evaluation of the cocycle $\zeta$ is as in Proposition~\ref{prop:cocyclevariation}
at the cochain in $C^1(T^2; \mathfrak{b}_n, \mathfrak{b}_n')$ is:
$$
 \tr( \Re(b)\Im(\dot a) -\Re(a)\Im(\dot b) )\, ,
$$
where $\Re$ and $\Im$ denote the usual real and imaginary part of the coefficients.
\end{corollary}

\begin{proof}
By  Lemma~\ref{lemma:evaltorus} and Proposition~\ref{prop:cocyclevariation},  the evaluation of $\zeta$ is 
$$
i(\tr(pr_{i\fh_n}(a)pr_{\fh_n}(\dot b))-pr_{i\fh_n}(b)pr_{\fh_n}(\dot a))
$$
Let $pr_{\mathfrak{h}+i\mathfrak{h} }\colon \mathfrak{b}_n\to\mathfrak{h}+i\mathfrak{h}  $ denote the projection to the diagonal part,
then, as $  \mathfrak{h}\subset\mathfrak{su}(n)$ is the subalgebra of diagonal matrices with zero real part,
$$
 pr_{\mathfrak{h}}=  i\, \Im\circ pr_{\mathfrak{h}+i\mathfrak{h} } \qquad  pr_{i\, \mathfrak{h}}=  \Re\circ pr_{\mathfrak{h}+i\mathfrak{h} }
$$
Thus 
\begin{eqnarray*}
i \tr(pr_{i\fh_n}(a)pr_{\fh_n}(\dot b)-pr_{i\fh_n}(b)pr_{\fh_n}(\dot a))& =&  i \tr( \Re(a) i \Im (\dot b)- \Re(b) i \Im (\dot a) )
\\
& = &
-  \tr( \Re(a) \Im (\dot b)- \Re(b)  \Im (\dot a) ).
\end{eqnarray*}
 \end{proof}

This concludes the proof of the main theorem.

\subsection{Comparison with other variation formulas}
\label{subsection:n=2}

When $n=2$, we can write  $$a=\begin{pmatrix}
                   \frac{l_1+i\theta_1}2 & 0 \\
                   0 & -\frac{l_1+i\theta_1}2
                  \end{pmatrix}
                  \qquad \textrm{ 
                  and }\qquad
                  b=\begin{pmatrix}
                   \frac{l_2+i\theta_2}2 & 0 \\
                   0 & -\frac{l_2+i\theta_2}2
                  \end{pmatrix}.
                  $$
Hence $\exp(a)$ is an hyperbolic isometry with translation length $l_1$ and rotation angle $\theta_1$, and so is
$\exp(b)$  with parameters $l_2$ and $\theta_2$. Then,
by Corollary~\ref{corollary:formula}, the contribution to 
the variation of volume of the corresponding torus component is
$$
\tr(\Re(b)\Im(\dot a)- \Re(a)\Im(\dot b))
     = \frac{1}{2}(l_2\dot \theta_1-l_1\dot\theta_2),
$$
which is precisely  Hodgson's formula in \cite{Hodgson}, as he derived from 
 Schl\"afli's formula for the  variation of the volume for polyhedra in hyperbolic space.

Still in the case $n=2$ Neumann and Zagier \cite{NZ} study the space of hyperbolic structures on a manifold by studying triangulations by ideal hyperbolic simplices.
To each hyperbolic ideal triangulation there is a  natural assignment of a holonomy representation in $\operatorname{PSL}_2(\C)$, and 
its volume is then  just the addition of the volumes of the tetrahedra involved. 

For an arbitrary value of $n$, variational formulas for the volume have been obtained in remarkable work by several authors  using spaces of decorated ideal triangulations and the Bloch group, see for instance \cite{GTZ}. Here we shall briefly describe the approach of \cite{BFG} and \cite{DGG} and relate their formulas to ours.

For $n=3$, Bergeron-Falbel-Guilloux \cite{BFG} consider ideal hyperbolic tetrahedra  with an additional decoration by flags 
in $\mathbb P^2(\C)$ (see also  \cite{GTZ}). Under some compatibility conditions one gets back  the manifold  equipped with a decorated hyperbolic structure, to which
one can associate a holonomy in $\operatorname{PSL}_3(\C)$, as well as a flag to each peripheral subgroup (equivalently this fixes yields a Borel subgroup for the holomomy of each peripheral subgroup). Pushing this data to the Bloch group gives then a volume for the holonomy. 

Firstly the volume in \cite{BFG} is $1/4$ of ours, they chose a normalization of the volume such that composing with the irreducible representation 
$\sigma_3\colon \operatorname{SL}_2(\C)\to \operatorname{SL}_3(\C)$ does not change the volume (in our case, by Proposition~\ref{prop:volsigman} it is multiplied by $4$).
Secondly, they have a different choice of   coordinates in $\operatorname{PSL}_3(\C)$: the holonomy of the peripheral elements $m$ and $l$ is, given respectively by, 
\begin{equation}
\label{eqn:holonomyPGL3}  
\begin{pmatrix}
 \frac{1}{A^*} & * & * \\
 0 & 1 & * \\
 0 & 0 & A
\end{pmatrix}
\quad\textrm{ and } \quad
\begin{pmatrix}
 \frac{1}{B^*} & * & * \\
 0 & 1 & * \\
 0 & 0 & B
\end{pmatrix},
\end{equation}
\cite[\S 5.5.2]{BFG}.
Then Proposition 11.1.1 of \cite{BFG} states that each end contributes to the variation of volume by a term
\begin{equation}
\label{eqn:BFG}  
\frac{1}{12}\Im(d\log\wedge_\Z\log)(2\, A\wedge _\Z B + 2\, A^*\wedge _\Z B^* +A^*\wedge _\Z B  + A\wedge _\Z B^* ) ,
\end{equation}
where  $\wedge_\Z$ stands for the wedge product as $\Z$-modules of the space of analytic functions on the space of decorated structures, and
\begin{equation}
\label{eqn:wedgeZ}
 \Im(d\log\wedge_\Z\log) (f \wedge _\Z g) = \Im\big(\log \vert g \vert \cdot d(\log f) -\log \vert f \vert \cdot d(\log g) \big)
\end{equation}
for any pair of analytic functions $f$ and $g$. Then, after a change of coordinates in  $\operatorname{PSL}_3(\C)$,
it is  straightforward to check that  \eqref{eqn:BFG} is $1/4$ of  Corollary~\ref{corollary:formula} for $\operatorname{SL}_3(\C)$.


When $n\geq3$, Dimofte, Gabella, and Goncharov in \cite{DGG} also consider the space of framed flat connection. Thies yields  decorated ideal triangulations by means of flags in 
$\mathbb P^{n-1}(\C)$ and they generalize Equation~\eqref{eqn:BFG}. In their work then,  the holonomy of the peripheral elements $l$ and $m$ (resp. $a$ and $b$ in our setting ) is given by  
\begin{equation*}
 \begin{pmatrix}
  1 & 0   &    0   &         & 0\\
  * & l_1 &    0   &         & 0\\
  * & *   & l_1l_2 &         & 0\\
    &     &        & \ddots  & \\
  *  &	*  &   *     &         & l_1\dots l_{n-1}
 \end{pmatrix}
\quad \textrm{ and } \quad 
 \begin{pmatrix}
  1 & 0   &    0   &         & 0\\
  * & m_1 &    0   &         & 0\\
  * & *   & m_1m_2 &         & 0\\
    &     &        & \ddots  & \\
  *  &	*  &   *     &         & m_1\dots m_{n-1}
 \end{pmatrix},
\end{equation*}
 \cite[(3.42)]{DGG}.
 If one denotes by $\kappa$ the Cartan matrix of size $n-1$ given by
\begin{equation*}
\kappa_{ij}=\begin{cases}
2  & \text{for } i=j,\\
-1 & \text{for } i=j\pm 1, \\
0 & \text{otherwise.}
\end{cases}
\end{equation*} 
then the contribution of each peripheral group to the variation of volume is then (\cite[(4.52) and (4.53)]{DGG}):
\begin{equation}
 \label{eqn:DGG}
 \log d\arg \sum_{i,j=1}^n(\kappa^{-1})_{ij}l_i\wedge m_j.
\end{equation}
Here (\cite[4.60]{DGG}):
\begin{equation*}
\log d\arg (f\wedge g)= \log |f| d\arg g-\log |g| d\arg f 
\end{equation*}
  is the exact
 the analog of \eqref{eqn:wedgeZ}.

Again, an easy computation shows that \eqref{eqn:DGG} is the same formula as Corollary~\ref{corollary:formula}.

As conclusion, our work gets back exactly the same formula as in \cite{BFG} and \cite{DGG} but with the advantage that we do not have to bother 
about the existence of decorated ideal triangulations (the existence of non-degenerate ideal triangulations for the complete structure
 still remains conjectural).

%

\bibliographystyle{plain}\label{biblography}

\end{document}